\RequirePackage{fix-cm}
\documentclass[smallextended]{svjour3}       % onecolumn (second format)
\smartqed  % flush right qed marks, e.g. at end of proof
\usepackage{graphicx}
\usepackage{subfig}
\usepackage{color}
\usepackage{amssymb,amsmath}
\newcommand{\abs}[1]{\vert#1\vert}
\newcommand{\babs}[1]{\big\vert#1\big\vert}

\newcommand{\Abs}[1]{\Vert#1\Vert}
\newcommand{\bAbs}[1]{\big\Vert#1\big\Vert}
\newcommand{\grad}{\nabla_{X,z}}

\newcommand{\dsp}{\displaystyle}
\newcommand{\dt}{\partial_t}
\newcommand{\dz}{\partial_z}
\newcommand{\R}{{\mathbb R}}
\newcommand{\cS}{{\mathcal S}}
\newcommand{\eps}{\varepsilon}
\newcommand{\rhob}{{\underline{\rho}}}
\newcommand{\N}{{\mathbb N}}
\newcommand{\bff}{{\bf f}}

\newcommand{\rhoeq}{\rho_{\rm eq}}

\begin{document}

\title{Normal mode decomposition and dispersive and nonlinear mixing in stratified fluids%\thanks{Grants or other notes
}
%\subtitle{Internal waves in stratified fluids}

\titlerunning{Internal waves in stratified fluids}        % if too long for running head

\author{Beno\^{i}t Desjardins        \and
        David Lannes \and
        Jean-Claude Saut %etc.
}


\institute{B. Desjardins \at
              CMLA ENS Cachan, 61 Avenue du Président Wilson, 94230 Cachan, France \\
              %\email{fauthor@example.com}           %  \\
           \and
           D. Lannes \at
              Universit\' e de Bordeaux, IMB et CNRS UMR 5251,  33405 Talence , France
\email{david.lannes@math.u-bordeaux.fr}  
\and
J.-C.Saut \at
Laboratoire de Math\' ematiques, UMR 8628,
Universit\' e Paris-Sud et CNRS, Bat. 307, 91405 Orsay, France
}

\date{Received: date / Accepted: date}

\maketitle

{\quote \it To the memory of our friend Walter Craig}

\bigbreak

\begin{abstract}
Motivated by the analysis of the propagation of internal waves in a stratified ocean, we consider in this article the incompressible Euler equations with variable density in a flat strip, and we study the evolution of perturbations of  the hydrostatic equilibrium corresponding to a stable vertical stratification of the density. We show the local well-posedness of the equations in this configuration and provide a detailed study of their linear approximation. Performing a modal decomposition according to a Sturm-Liouville problem associated to the background stratification, we show that the linear approximation can be described by a series of dispersive perturbations of linear wave equations. When the so called Brunt-Vais\"al\"a frequency is not constant, we show that these equations are coupled, hereby exhibiting a phenomenon of dispersive mixing. We then consider more specifically shallow water configurations (when the horizontal scale is much larger than the depth); under the Boussinesq approximation (i.e. neglecting the density variations in the momentum equation), we provide a well-posedness theorem for which we are able to control the existence time in terms of the relevant physical scales. We can then extend the modal decomposition to the nonlinear case and exhibit a nonlinear mixing of different nature than the dispersive mixing mentioned above. Finally, we discuss some perspectives such as the sharp stratification limit that is expected to converge towards two-fluids systems.
\end{abstract}

\section{Introduction}

\subsection{General setting}
The aim of this paper is to analyze the propagation of internal waves in a continuously stratified fluid. 
%More precisely those waves may appear when the density of the (incompressible) fluid varies only within a layer whose thickness $h$ is much smaller than  the total depth and $h$ must be considered as the fundamental scale against which wave amplitude and length are to be measured. 
In oceans, this variation of density  (pycnocline), may be due to a difference of salinity (halocline) or temperature (thermocline); note also that such waves also appear in applications to atmospheric studies (e.g. \cite{Flierl,Feliks,Klein}).

The modeling of such waves has a long history, starting with the pioneering mathematical  study of periodic waves by  Dubreil-Jacotin \cite{DJ}. We refer to  \cite{Ben2,BC,DA,Long,Long2,Yih,Yih2} for theoretical and experimental studies of solitary waves and to \cite{HM} for   a survey on oceanic internal waves. %On the other hand we are not aware of rigorous derivations of asymptotic models of internal waves in a continuously stratified fluid.

While many classical equations such that the Benjamin-Ono equation \cite{Ben,Ono} (see also \cite{DA})  or the Intermediate Long Wave equation \cite{KKD} have been formally derived in this context, no rigorous derivation seems to be available. This is in contrast with the two-layers formulation where the internal waves propagate at the interface of two layers of (incompressible) fluids of different densities. In this setting one can generalize the classical approach for surface waves (see {\it eg} \cite{Lannes_book}) and derive rigorously (in the sense of consistency), in various regimes, a plethora of asymptotic models including all the classical models of internal waves. Roughly speaking this is achieved by expanding with respect to suitable small parameters two non-local operators that appear when expressing the (free boundary), two-layers system as an equation on a fixed domain. Together with the delicate analysis of the Cauchy problem for the full two-layers system (see \cite{Lannes2}), which involves Kelvin-Helmholtz type instabilities (see also \cite{BarrosChoi,BarrosChoi2,LannesMing} for the persistence of these instabilities in shallow water asymptotic models),  this leads to the complete justification of some internal waves asymptotic systems.

The situation is quite different for a continuously stratified fluid. There is no more free boundary and one starts from the non-homogeneous Euler system for which the local well-posedness of the Cauchy problem is well known (see {\it eg} \cite{Danchin} and the references therein) although it does not seem to have been considered in the present setting. Another difficulty, addressed here, is to establish a time of existence for the solutions which is relevant with respect to the different physical scales involved.

From a more qualitative viewpoint, it is common in oceanography or atmospheric studies to decompose the various quantities of interest on a well chosen basis of vertical modes related to the background stratification (e.g. \cite{Flierl,Gill}). This approach has not been fully justified so far; we provide such a rigorous justification here, exhibiting additional conditions that need to be satisfied if one wants this decomposition to converge properly. Moreover, in most of the studies, the non-hydrostatic component of the pressure is neglected; we show here how to take it into account and that, at the linear level, its contribution is of dispersive nature. We also show that in some cases (when the so-called Brunt-Vais\"al\"a frequency is not constant), this dispersive term induces some mixing between the different modes of the decomposition. In the case of a constant Brunt-Vais\"al\"a frequency where such mixing does not occur at leading (linear) order, we show that such mixing occur in shallow water at the next order due to  nonlinearities, and we derive a sequence of coupled Boussinesq-like systems.

\bigbreak

We now make more  precise the physical context of our study.

\vspace{0.3cm}
We assume that the fluid domain is infinite in the horizontal direction $X\in \R^d$ ($d=1,2$) 
delimited by a flat bottom located at $z=-H$ and a rigid lid at $z=0$. The velocity field at time $t$ and at the point $(X,z)$ of the fluid domain is denoted by $U(t,X,z)\in \R^{d+1}$ and its horizontal and vertical components are respectively denoted by $V(t,X,z)\in \R^d$, $w(t,X,z)\in \R$. We also denote by $P(t,X,z)\in \R$ the pressure field and by ${\bf g}=-g{\bf e}_z$ the (constant) acceleration of gravity. The Euler equations governing the fluid motion are therefore
\begin{equation}
\label{Euler}
\left\lbrace
\begin{array}{l}
\dsp \rho(\dt U+U\cdot\grad U)=-\grad P+\rho {\bf g},\\
\dsp \dt \rho+U\cdot\grad \rho=0,\\
\dsp \nabla_{X,z}\cdot U=0
\end{array}\right.
\qquad (t\geq 0, X\in \R^d, z\in (-H,0)),
\end{equation}
with the boundary conditions
\begin{equation}\label{BC}
w_{\vert_{z=-H}}=w_{\vert_{z=0}}=0
\end{equation}
expressing the impermeability of the rigid bottom and lid.

These equations possess equilibrium solutions $(U_{\rm eq},\rho_{\rm eq},P_{\rm eq})$ depending only on the vertical variable $z$ of the form
\begin{equation}\label{equil}
U_{\rm eq}=0, \quad \rho_{\rm eq}=\rho_{\rm eq}(z),\quad \frac{d}{dz} P_{\rm eq}=-\rho_{\rm eq}g.
\end{equation}
Perturbation of such equilibrium solutions give rise in the oceans to the propagation of waves called ``internal waves''. These internal waves are therefore exact solutions $(U_{ex},\rho_{ex},P_{ex})$ to 
 (\ref{Euler})-(\ref{BC}) of the form
$$
U_{\rm eq}=\eps U, \quad \rho_{ex}=\rho_{\rm eq}+\eps \rho ,\quad P_{ex}=P_{\rm eq}+\eps P,
$$
where $\eps>0$ is a parameter measuring the amplitude of the perturbation, and
with $(U,\rho,P)$ solving
\begin{equation}
\label{Euler_pert}
\left\lbrace
\begin{array}{l}
\dsp (\dt U+\eps U\cdot\grad U)=-\frac{1}{\rho_{\rm eq}+\eps \rho}\grad P-\frac{\rho}{\rho_{\rm eq}+\eps \rho} g{\bf e_z},\\
\dsp \dt \rho+\eps U\cdot\grad \rho+w\frac{d}{dz}\rho_{\rm eq}=0,\\
\dsp \nabla_{X,z}\cdot U=0
\end{array}\right.
\end{equation}
with the boundary conditions
\begin{equation}\label{BC_pert}
w_{\vert_{z=-H}}=w_{\vert_{z=0}}=0.
\end{equation}

\vspace{0.3cm}
The paper is organized as follows. In Section \ref{sectWP} we establish the local well-posedness for the Euler system in the configuration \eqref{Euler_pert}-\eqref{BC_pert}  considered here.
We then consider in Section \ref{sectlin}  the linear approximation to the full system of equations \eqref{Euler_pert}-\eqref{BC_pert}.  This approximation is justified in \S \ref{secterrlin} and the normal mode decomposition of the solutions to this linear system  is performed in  \S \ref{SectSL}. When the so-called Brunt-Vais\"al\"a frequency is constant, we show in \S \ref{sectNconst} that the evolution of the coefficients of this decomposition is governed by a sequence of uncoupled dispersive perturbations of waves equations. When the Brunt-Vais\"al\"a frequency is not constant, we exhibit in \S \ref{sectNpasconst} the mechanism of dispersive mixing. Particular attention is also paid to the derivation of additional conditions ensuring a proper convergence of the modal decomposition. 

Finally, in Section \ref{sectSW} we consider the case of shallow water configurations, when the horizontal scale of the perturbations is much larger than the depth of the ocean.  Under the additional {\it strong Boussinesq assumption} under which the density is assumed to be constant in Euler's equations, we are able to derive nonlinear models. The first step is to derive in  \S \ref{sectunif} a local existence theorem for the non-dimensionalized system that ensures that the existence time is relevant with respect to the different physical scales of the problem. We then extend in \S \ref{sectmodeNL} the modal representation introduced in Section \ref{sectlin} in order to take into account the nonlinear effect. It is shown that they induce another kind of mixing between the modes. Finally, some perspectives are considered in Section \ref{sectperspectives}, such as the sharp stratification limit towards two-fluids models.

\medbreak

\subsection{Notations}

- $X=(x,y)\in \R^2$ denotes the horizontal variables. We also denote  by $z$ the vertical variable.\\
- $\nabla$ is the gradient with respect to the horizontal variables;
$\nabla_{X,z}$ is the full three dimensional gradient operator. \\
- We denote by $d=1,2$ the horizontal dimension. When $d=1$, we often
identify functions on $\R$ as functions on $\R^2$ independent of the
$y$ variable. In particular, when $d=1$, the gradient operator takes the form
$$
\nabla_{X,z}f=\left(\begin{array}{c}\partial_x f \\ 0 \\ \partial_z
    f \end{array}\right).
$$
- $\cS$ is the flat strip $\R^d\times (-H,0)$ (or $\R^d\times (-1,0)$ when working with dimensionless variables in Section \ref{sectSW}).\\
- We write $U$ the velocity field; its horizontal component
 is written ${ V}$, and its vertical component
${w}$.\\
- We always use simple bars to denote functional norms on $\R^d$ and
double bars to denote functional norms on the $d+1$ dimensional
domain $\cS$; for instance
$$
\abs{f}_p=\abs{f}_{L^p(\R^d)},\quad
\abs{f}_{H^s}=\abs{f}_{H^s(\R^d)},\quad
\Abs{f}_p=\Abs{f}_{L^p(\cS)}, \mbox{ etc.}
$$
- For $f,g\in L^2(\cS)$, we denote by $(f,g)$ the standard $L^2(\cS)$ scalar product.\\
- We use the Fourier multiplier notation
$$
f(D)u={\mathcal F}^{-1}(\xi\mapsto f(\xi)\widehat{u}(\xi))
$$
and denote by $\Lambda=(1-\Delta)^{1/2}=(1+\abs{D}^2)^{1/2}$ the
fractional derivative operator.\\
-We define, for
all $s\in \R$, $k\in \N$ the space $H^{s,k}=H^{s,k}(\cS)$ by
\begin{equation}\label{defHsk}
H^{s,k}=\bigcap_{j=0}^k H^j((-H,0);H^{s-j}(\R^d)),\quad \mbox{ with }\quad
\Abs{u}_{H^{s,k}}=\sum_{j=0}^k \Abs{\Lambda^{s-j}\dz^j u}_2.
\end{equation}
- If $\omega$ is a positive scalar function, we denote by $L^2_\omega$ the weighted $L^2$-space on $\cS$ with associated norm
$$
\Abs{f}_{L^2_\omega}^2 =\int_{\cS} \vert f(X,z) \vert^2 \omega(X,z) {\rm d}X{\rm d}z.
$$
- We generically denote by $C(\cdot)$ some positive function that has
a nondecreasing dependence on its arguments.\\

\section{Local well-posedness}\label{sectWP}

The Cauchy problem for the non-homogeneous Euler equations has been
considered in various settings: whole space or bounded and unbounded
domains, $L^2$ or $L^p$ based spaces, etc (see for instance
\cite{Itoh,ItohTani,Danchin}). It seems however that the Cauchy
problem for the configuration considered here (unbounded domain with a
density whose gradient is not in $L^2(\cS)$) is not included in
existing results. We therefore provide
below a local well-posedness result. Contrary to the above references,
we do not seek in this result to
be sharp in the regularity requirements for the initial conditions;
our main concern is rather to  distinguish as much as possible the
vertical and horizontal derivatives in the proof, since they play drastically different roles in the qualitative descriptions of the solutions addressed in this paper.
\begin{theorem}\label{theorem1}
Let $0<\eps\leq 1$, $\nu > d +2$, $\rho_{\rm eq}\in W^{\nu,\infty}(-H,0)$, and $\rho^0, U^0\in H^{\nu}(\cS)$ be such that
$\grad\cdot U^0=0$ and 
$$
\begin{cases}
\exists \rho_{\rm min}>0,\qquad \inf_{z\in [-H,0]} \rho_{\rm eq}(z)\geq \rho_{\rm
  min}\quad\mbox{ and }\quad \inf_{\cS}(\rho_{\rm eq}+\eps \rho^0)\geq
\rho_{\rm min},\\
\exists M_0>0, \quad \forall z\in (-H,0), \qquad \frac{1}{M_0}\leq -\rho'_{\rm eq}(z)\leq M_0.
\end{cases}
$$
Then there exists $T>0$ such that for all $\eps\in (0,1]$, there is a unique solution $(U,\rho)\in
C([0,T];H^{\nu}(\cS)^{d+2})$ to
\eqref{Euler_pert}-\eqref{BC_pert} with initial condition
$(U^0,\rho^0)$; moreover,
$$
\frac{1}{T}=c_1,\qquad \sup_{t\in [0,T]}
\big(\Abs{U}_{H^\nu}+\Abs{\rho}_{H^\nu}\big)\leq c_2,
$$
with $c_j=C\big(\frac{1}{\rho_{\min}},M_0,\Abs{\rho_{\rm
    eq}}_{W^{\nu,\infty}},H,\Abs{\rho^0}_{H^{\nu}},\Abs{U^0}_{H^\nu}\big)$,
$j=1,2$. 
\end{theorem}
\begin{remark}
The existence time provided by the theorem is independent of
$\eps$. Note however that without the linear terms, one would obtain
an existence time of order $O(1/\eps)$ instead of $O(1)$. The issue of large time existence (as well as shallow water stability) is addressed in Theorem \ref{theorem3} below.
\end{remark}
\begin{proof}
We just derive here a priori estimates on solutions to
\eqref{Euler_pert}-\eqref{BC_pert}; the construction of solutions from
these energy estimates being obtained with classical means. As for the
standard Euler equation, the key point is to control the pressure
term which, owing to the fact that $U$ is divergence free, is given by
the resolution of the following boundary value problem,
\begin{equation}\label{eqpression}
\left\lbrace
\begin{array}{l}
\dsp -\grad\cdot \frac{1}{\rho_{\rm eq}+\eps \rho}\grad P=\eps \grad \cdot
\big(U\cdot \grad U\big) 
+\dz \Big(\frac{\rho g}{\rho_{\rm eq}+\eps \rho}\Big),\\
\dsp -\frac{1}{\rho_{\rm eq}+\eps \rho}\dz P_{\vert_{z=-H,0}}=\Big(\frac{\rho g}{\rho_{\rm eq}+\eps \rho}\Big)_{\vert_{z=-H,0}}.
\end{array}\right.
\end{equation}
Existence of solutions to \eqref{eqpression} follows classically from
Lax-Milgram's theorem. We provide below the $H^\nu$ estimates on $\grad
P$ that we shall need to establish the a priori estimates on \eqref{Euler_pert}-\eqref{BC_pert}.
\begin{lemma}\label{lemmpress}
Let $\nu > d+2$, and $\rho, U\in H^{\nu}(\cS)$. If moreover
$\grad\cdot U=0$ and if
$$
\exists \rho_{\rm min}>0,\qquad \inf_{z\in (-H,0)} \rho_{\rm eq}(z)\geq \rho_{\rm
  min}\quad\mbox{ and }\quad \inf_{\cS}(\rho_{\rm eq}+\eps \rho)\geq
\rho_{\rm min},
$$
then the solution to \eqref{eqpression} satisfies the estimate
$$
\Abs{\grad P}_{H^{\nu}}\leq C(\frac{1}{\rho_{\min}},\Abs{\rho_{\rm
    eq}}_{W^{\nu,\infty}},\Abs{\rho}_{H^{\nu}})\big(\Abs{\rho}_{H^{\nu}}+\eps\Abs{U}_{H^{\nu}}^2\big).
$$
\end{lemma}
\begin{proof}
{\it Since we work in a domain with boundaries, we need to distinguish
horizontal and vertical derivatives; to this purpose, we use here the
$H^{s,k}$ spaces introduced in \eqref{defHsk}, and we repeatedly use the
continuous embedding $H^{s+1/2,1}\subset
L^\infty((-H,0);H^{s}(\R^d))$ (see for instance Proposition 2.12 in \cite{Lannes_book}).}\\
Let us consider first  the following more general boundary value problem, with ${\bf
  f}\in C([-H,0];L^2(\cS)^3)$ such that ${\bf
  f}_{\vert_{z=-H,0}}=0$, $g\in L^2(\cS)$ and $h\in
C([-H,0];L^2(\R^d))$,
\begin{equation}\label{pbbase}
\left\lbrace
\begin{array}{l}
\dsp -\grad\cdot \frac{1}{\rho_{\rm eq}+\eps \rho}\grad Q=\eps \grad \cdot
{\bf f} +\eps \abs{D} g
+\dz h,\\
\dsp -\frac{1}{\rho_{\rm eq}+\eps \rho}\dz Q_{\vert_{z=-H,0}}=h_{\vert_{z=-H,0}}.
\end{array}\right.
\end{equation}
Multiplying by $Q$ and integrating by parts in both sides of the
equation (note that the boundary terms cancel  each other), we get
$$
\big(\frac{1}{\rho_{\rm eq}+\eps\rho}\grad Q,\grad Q\big)=-\eps \big({\bf
  f},\grad Q \big)+\eps (g,\abs{D} Q)- (h,\dz Q),
$$
from which one readily gets, with $\rho_{\rm max}=\Abs{\rho_{\rm eq}+\eps\rho}_\infty$,
\begin{equation}\label{estimL2}
\frac{1}{\rho_{\rm max}}\Abs{\grad Q}_2\leq \eps \Abs{{\bf f}}_2 +\eps \Abs{g}_2+\Abs{h}_2.
\end{equation}
Applying this estimate with $Q=P$, ${\bf f}=U\cdot \grad U$, $g=0$ and
$h=\frac{\rho g}{\rho_{\rm eq}+\eps \rho}$, we get the following
$L^2$-control on the solution $\grad P$ to \eqref{eqpression},
\begin{equation}\label{estimL2P}
\frac{1}{\rho_{\rm max}}\Abs{\grad P}_2\leq \eps \Abs{U\cdot \grad U}_2 +\Abs{\frac{\rho g}{\rho_{\rm eq}+\eps \rho}}_2.
\end{equation}
Horizontal derivatives of $\grad P$ can also be controlled in $L^2$ by
applying $\Lambda^{r-1}\abs{D}$ ($0\leq r\leq s$ and $s >d/2+3/2$) to \eqref{eqpression} and using
\eqref{pbbase} with $Q=\Lambda^{r-1}\abs{D}P$, ${\bf
  f}=\frac{1}{\eps}[\Lambda^{r-1}\abs{D},\frac{1}{\rho_{\rm eq}+\eps \rho}]\grad Q$,
$g=\Lambda^{r-1}\grad\cdot (U\cdot \grad U)$, and
$h=\Lambda^{r-1}\abs{D}\Big(\frac{\rho g}{\rho_{\rm eq}+\eps
  \rho}\Big)$. The estimate \eqref{estimL2} then yields
\begin{align}
\nonumber
\frac{1}{\rho_{\rm max}}\Abs{\Lambda^{r-1}&\abs{D}\grad P}_2\leq \eps
\Abs{\frac{1}{\eps}[\Lambda^{r-1}\abs{D},\frac{1}{\rho_{\rm eq}+\eps
    \rho}]\grad P}_2\\
\label{estimhorP}
&+\eps \Abs{\Lambda^{r-1}\grad\cdot (U\cdot \grad U)}_2 +\Abs{\Lambda^{r-1}\abs{D}\Big(\frac{\rho g}{\rho_{\rm eq}+\eps
  \rho}\Big)}_2;
\end{align}
we now turn to control the three terms in the right-hand-side:\\
- {\it Control of $\Abs{\frac{1}{\eps}[\Lambda^{r-1}\abs{D},\frac{1}{\rho_{\rm eq}+\eps
    \rho}]\grad P}_2$}. Remarking that $\frac{1}{\rho_{\rm eq}+\eps
    \rho}=\frac{1}{\rho_{\rm eq}}- \frac{\eps\rho}{\rho_{\rm eq}(\rho_{\rm eq}+\eps
    \rho)}$ and that $[\Lambda^{r-1}\abs{D},\frac{1}{\rho_{\rm
      eq}}]=0$, we have
\begin{align*}
\bAbs{\frac{1}{\eps}[\Lambda^{r-1}\abs{D},\frac{1}{\rho_{\rm eq}+\eps
    \rho}]\grad
  P}_2&=\bAbs{[\Lambda^{r-1}\abs{D},\frac{\rho}{\rho_{\rm eq}(\rho_{\rm eq}+\eps
    \rho)}]\grad P}_2\\
&\leq C(\frac{1}{\rho_{\rm min}},\Abs{\rho}_{W^{1,\infty}})\Abs{\rho}_{H^{s,1}}
\Abs{\grad P}_{H^{r-1/2,1}},
\end{align*}
where we used the commutator estimate \eqref{estcom} and the
assumption that $s>d/2+3/2$ to derive the
second inequality.\\
- {\it Control of $\Abs{\Lambda^{r-1}\grad\cdot (U\cdot \grad U)}_2
  $}. Recalling that $\grad\cdot U=0$, we can write classically
$$
\grad\cdot \big(U\cdot \grad U\big)=\sum_{i,j=x,y,z} \partial_i
U_j \partial_j U_i,
$$
and therefore, using the product estimate \eqref{estprod} and the
assumption that $s >d/2+3/2$ (recall also that $0\leq r\leq s$),
$$
\Abs{\Lambda^{r-1}\grad\cdot (U\cdot \grad U)}_2\leq
\Abs{U}_{H^{s,2}}^2.
$$
- {\it Control of $\Abs{\Lambda^{r-1}\abs{D}\Big(\frac{\rho g}{\rho_{\rm eq}+\eps
  \rho}\Big)}_2$.} Without any difficulty, one gets
$$
\Abs{\Lambda^{r-1}\abs{D}\Big(\frac{\rho g}{\rho_{\rm eq}+\eps
  \rho}\Big)}_2\leq C(\frac{1}{\rho_{\rm min}},\Abs{\rho}_{W^{1,\infty}}) \Abs{\rho}_{H^{s,1}}.
$$
\medbreak
Gathering these three estimates, together with \eqref{estimL2P} and
\eqref{estimhorP}, we finally get
\begin{align*}
\Abs{\grad P}_{H^{r,0}}\leq &C(\frac{1}{\rho_{\min}},\Abs{\rho_{\rm
    eq}}_\infty,\Abs{\rho}_{W^{1,\infty}}) \\
   & \times \Big( \Abs{\rho}_{H^{s,1}}+\eps\Abs{U}_{H^{s,2}}^2+\eps\Abs{\grad P}_{H^{r-1/2,1}}\Big),
\end{align*}
for all $0\leq r\leq s$. We now prove that it is possible to obtain an estimate on the $H^{r,1}$-norm instead of the
$H^{r,0}$ norm in the left-hand-side, at the cost of increasing
slightly the regularity in $z$ of $\rho$. Using the equation satisfied by $P$; one has
\begin{align}
\nonumber
-\dz^2 P&=(\rho_{\rm eq}+\eps\rho)\dz\big(\frac{1}{\rho_{\rm eq}+\eps
  \rho}\big)\dz P+(\rho_{\rm eq}+\eps\rho)\nabla\cdot
\frac{1}{\rho_{\rm eq}+\eps\rho}\nabla P\\
\label{dz2P}
&+\eps (\rho_{\rm eq}+\eps\rho)\sum_{i,j=x,y,z} \partial_i
U_j \partial_j U_i,
+(\rho_{\rm eq}+\eps\rho)\dz \Big(\frac{\rho g}{\rho_{\rm eq}+\eps \rho}\Big);
\end{align}
controlling the right-hand-side in $H^{r-1,0}$ through \eqref{estprod}
shows that $\Abs{\Lambda^{r-1}\dz^2 P}_{H^{r,0}}$  is bounded from above by 
$$
C(\frac{1}{\rho_{\min}},\Abs{\rho_{\rm
    eq}}_{W^{1,\infty}},\Abs{\rho}_{W^{1,\infty}})\Big(\Abs{\rho}_{H^{s,2}}+\eps\Abs{U}_{H^{s,2}}^2+\eps\Abs{\grad P}_{H^{r-1/2,1}}\Big),
$$
and we have therefore obtained that
\begin{align*}
\Abs{\grad P}_{H^{r,1}} \leq &C(\frac{1}{\rho_{\min}},\Abs{\rho_{\rm
    eq}}_{W^{1,\infty}},\Abs{\rho}_{W^{1,\infty}}) \\
    &\times \Big(\!\Abs{\rho}_{H^{s,2}}+\eps\Abs{U}_{H^{s,2}}^2+\eps\Abs{\grad P}_{H^{r-1/2,1}}\!\Big);
\end{align*}
by a simple
finite induction on $r$, we therefore get
\begin{equation}\label{controP1}
\Abs{\grad P}_{H^{s,1}}\leq C(\frac{1}{\rho_{\min}},\Abs{\rho_{\rm
    eq}}_{W^{1,\infty}},\Abs{\rho}_{W^{1,\infty}})\big(\Abs{\rho}_{H^{s,2}}+\eps\Abs{U}_{H^{s,2}}^2\big).
\end{equation}
Using the product estimate \eqref{prodalg}, we also get from
\eqref{dz2P} that for all $0\leq k\leq \nu$,
\begin{align*}
\Abs{\grad P&}_{H^{\nu,k}}\leq \Abs{\grad P}_{H^{\nu,1}}+\Abs{\dz^2
  P}_{H^{\nu-1,k-1}}\\
&\leq  C\big(\frac{1}{\rho_{\min}},\Abs{\rho_{\rm
    eq}}_{W^{\nu,\infty}},\Abs{\rho}_{H^{\nu}}\big)(\Abs{\rho}_{H^{\nu}}+\eps \Abs{U}_{H^\nu}^2+\Abs{\grad
P}_{H^{\nu,k-1}})
\end{align*}
(we used here the assumption that $\nu > d+2$). The result then
follows from an induction on $k$ and \eqref{controP1}.
\end{proof}
Applying $\Lambda^{\nu-k}\dz^k$ ($k=0,1,2$) to the equations of
\eqref{Euler_pert}, we get
$$
\left\lbrace
\begin{array}{l}
\dsp (\dt \tilde U_k+\eps U\cdot\grad \tilde
U_k)=F_1+F_2,\\
\dsp \dt \tilde \rho_k+\eps U\cdot\grad \tilde\rho_k=f_1+f_2,\\
\dsp \nabla_{X,z}\cdot \tilde U_k=0
\end{array}\right.
$$
where we denoted $\tilde U_k=\Lambda^{\nu-k}\dz^k U$ and $\tilde
\rho_k=\Lambda^{\nu-k}\dz^k \rho$, and where
\begin{align*}
F_1&=-\Lambda^{\nu-k}\dz^k\Big(\frac{1}{\rho_{\rm eq}+\eps \rho}\grad
P\Big)-\Lambda^{\nu-k}\dz^k\Big(\frac{\rho}{\rho_{\rm eq}+\eps \rho} g{\bf
  e_z}\Big),\\
F_2&=-\eps(\Lambda^{\nu-k}\dz^k,U)\cdot \grad U
\end{align*}
and
\begin{align*}
f_1=-\Lambda^{\nu-k}\dz^k\Big(w\frac{d}{dz}\rho_{\rm eq}\Big), \qquad f_2=-\eps [\Lambda^{\nu-k}\dz^k,U]\cdot \grad \rho.
\end{align*}
Multiplying the first equation by
$\tilde U_k$ and the second one by $\tilde \rho_k$, and integrating by
parts, we get (recall that $\grad\cdot U=0$ and that $w$ vanishes at the boundaries), 
\begin{align*}
\frac{d}{dt}\Big(\frac{1}{2}\Abs{\tilde U_k}^2_2+\frac{1}{2}\Abs{\tilde\rho_k}_2^2\Big)&=(F_1+F_2,\tilde
U)+(f_1+f_2,\tilde \rho)\\
&\leq (\Abs{F_1}_2+\Abs{F_2}_2)\Abs{\tilde U_k}_2+(\Abs{f_1}_2+\Abs{f_2}_2)\Abs{\tilde \rho_k}_2
\end{align*}
We easily get from the product estimate \eqref{prodalg} and Lemma
\ref{lemmpress} (for $F_1$ and $f_1$) and the commutator estimate
\eqref{estcomf} (for $F_2$ and $f_2$) that
$$
\Abs{F_1}_2+\Abs{F_2}_2+\Abs{f_1}_2+\Abs{f_2}_2\leq C\big(\frac{1}{\rho_{\min}},\Abs{\rho_{\rm
    eq}}_{W^{\nu,\infty}},\Abs{\rho}_{H^{\nu}},\Abs{U}_{H^\nu}\big)
$$
and therefore
\begin{align*}
\frac{d}{dt}\Big(\Abs{\tilde
  U_k}_2^2+\Abs{\tilde\rho_k}_2^2\Big)&\leq  C\big(\frac{1}{\rho_{\min}},\Abs{\rho_{\rm
    eq}}_{W^{\nu,\infty}},\Abs{\rho}_{H^{n}},\Abs{U}_{H^\nu}\big)
\big(\Abs{\tilde U_k}_2+\Abs{\tilde \rho_k}_2),\\
&\leq C\big(\frac{1}{\rho_{\min}},\Abs{\rho_{\rm
    eq}}_{W^{\nu,\infty}},\Abs{\rho}_{H^{\nu}},\Abs{U}_{H^\nu}\big).
\end{align*}
Summing over all $0\leq k\leq \nu$, this yields
$$
\frac{d}{dt}\Big(\Abs{
  U}_{H^\nu}^2+\Abs{\rho}_{H^\nu}^2\Big)
\leq C\big(\frac{1}{\rho_{\min}},\Abs{\rho_{\rm
    eq}}_{W^{\nu,\infty}},\Abs{\rho}_{H^{\nu}},\Abs{U}_{H^\nu}\big),
$$
which is the desired a priori estimate.
\end{proof}
%\begin{remark}\label{remcomm}
%One can check that without the terms $F_1$ and $f_1$ in \eqref{systHO}, one would have a factor $\eps$ in front of the right-hand-side in \eqref{eqNRJ1} and therefore an existence time of size $O(1/\eps)$ instead of $O(1)$ in Theorem \ref{theorem1}. These two terms are due to the lack of commutation between the vertical differentiation operator $\dz$ and $\rho_{\rm eq}$ and $\rho'_{\rm eq}$. This obstruction can be overcome in the case where the Br\"unt-Vais\"al\"a frequency $N^2=g \frac{-\rho'_{\rm eq}}{\rho_{\rm eq}}$ does not depend on $z$, as shown below in Theorem \ref{theorem2}.
%\end{remark}
\section{A linear approximation}\label{sectlin}

We consider here the linear approximation to the full system of equations \eqref{Euler_pert}-\eqref{BC_pert}, formally obtained by setting $\eps=0$ in the equations. We show in \S \ref{secterrlin} that this approximation is as expected of precision $O(\eps)$ and then turn to analyze its behavior. As often in oceanography (e.g. \cite{Gill}), it is convenient to describe the vertical dependence of the functions involved in the problem by decomposing them on a Sturm-Liouville basis. General facts on Sturm-Liouville decompositions are recalled in \S \ref{SectSL}. We then show in \S \ref{sectNconst} that, when the so-called Brunt-Vais\"al\"a frequency $N$ is independent of $z$,  the coefficients of these decompositions are found by solving {\it uncoupled} dispersive perturbations of wave systems (whose speed are related to the eigenvalues of the Sturm-Liouville problem satisfied by the vertical velocity). We then turn to study in \S \ref{sectNpasconst} the general case where $N$ is not constant. We show that for such a configuration, the dispersive terms induce a coupling between the different modes of the decomposition. In both cases (constant and non constant $N$), we give sufficient conditions to improve the speed of convergence of the modal decomposition.

\subsection{Error estimate for the linear approximation}\label{secterrlin}

We consider here the linearized system obtained by taking $\eps=0$
in \eqref{Euler_pert}-\eqref{BC_pert}. Decomposing the equation on the
velocity into its horizontal and vertical components, this yields
\begin{equation}
\label{Euler_pert_lin}
\left\lbrace
\begin{array}{l}
\dsp \dt V +\frac{1}{\rho_{\rm eq}}\nabla P=0\\
\dsp \dt w +\frac{1}{\rho_{\rm eq}}\dz P+\frac{\rho}{\rho_{\rm eq}} g=0,\\
\dsp \dt \rho+w\frac{d}{dz}\rho_{\rm eq}=0,\\
\dsp \nabla_{X,z}\cdot U=0
\end{array}\right.
\end{equation}
with the boundary conditions
\begin{equation}\label{BC_pert_lin}
w_{\vert_{z=-H}}=w_{\vert_{z=0}}=0.
\end{equation}
As shown in the following proposition, the solutions of this linear model provide an $O(\eps)$ approximation to the exact solution of the full nonlinear problem \eqref{Euler_pert}-\eqref{BC_pert}. 
 \begin{proposition}
 Let the assumptions of Theorem \ref{theorem1} be satisfied and denote by $T>0$ the existence time of the solution $(U,\rho)$ of the nonlinear problem \eqref{Euler_pert}-\eqref{BC_pert} provided by this theorem. There exists also a unique solution 
$(U^{\rm lin},\rho^{\rm lin})\in
C([0,T];H^{\nu}(\cS)^{d+2})$ to the linear problem \eqref{Euler_pert_lin}-\eqref{BC_pert_lin} with same initial data, and the following error estimate holds
$$
\Vert (\rho-\rho^{\rm lin}, U-U^{\rm lin})\Vert_{L^\infty([0,T]\times H^{\nu-1})} \leq c_3 \eps
$$
with $c_3=C\big(T,\frac{1}{\rho_{\min}},M_0,\Abs{\rho_{\rm eq}}_{W^{\nu,\infty}},H,\Abs{\rho^0}_{H^{\nu}},\Abs{U^0}_{H^\nu}\big)$.
\end{proposition}
\begin{proof}
It is obvious that the solution to the linear system is defined on the same time interval as the solution to the nonlinear system provided by Theorem \ref{theorem1}. We need to prove the error estimate. Let us decompose the pressure $P$ in the nonlinear problem under the form $P=P^{\rm L} + \eps P^{\rm NL}$, 
with
$$
\left\lbrace
\begin{array}{l}
\dsp -\grad\cdot \frac{1}{\rho_{\rm eq}}\grad P^{\rm L}=\dz \Big(\frac{\rho g}{\rho_{\rm eq}}\Big),\\
\dsp -\frac{1}{\rho_{\rm eq}}\dz P^{\rm L}_{\vert_{z=-H,0}}=\Big(\frac{\rho g}{\rho_{\rm eq}}\Big)_{\vert_{z=-H,0}},
\end{array}\right.
$$
and
$$
\left\lbrace
\begin{array}{l}
\dsp -\grad\cdot \frac{1}{\rho_{\rm eq}+\eps \rho}\grad P^{\rm NL}=\grad \cdot F,\\
\dsp -\frac{1}{\rho_{\rm eq}+\eps \rho}\dz P^{\rm NL}_{\vert_{z=-H,0}}=F_{\rm v},
\end{array}\right.
$$
where 
$$
F=U\cdot \grad U
-\frac{\rho^2 }{\rho_{\rm eq}(\rho_{\rm eq}+\eps \rho)} g{\bf e_z} - \frac{\rho}{\rho_{\rm eq}(\rho_{\rm eq}+\eps \rho)}\grad P^{\rm L},
$$
and $F_{\rm h}$ and $F_{\rm v}$ denote its horizontal and vertical components.
The difference $\widetilde{U}=U-U^{\rm lin}$, $\widetilde{\rho}=\rho-\rho^{\rm lin}$ and $\widetilde{P}=P^{\rm L}-P^{\rm lin}$ satisfies therefore
$$
\left\lbrace
\begin{array}{l}
\dsp \dt \widetilde{V} +\frac{1}{\rho_{\rm eq}}\nabla \widetilde{P}=-\eps \frac{1}{\rho_{\rm eq}+\eps\rho}\nabla P^{\rm NL}+\eps F_{\rm h}\\
\dsp \dt \widetilde{w} +\frac{1}{\rho_{\rm eq}}\dz \widetilde{P}+\frac{\widetilde{\rho}}{\rho_{\rm eq}} g=-\eps \frac{1}{\rho_{\rm eq}+\eps\rho}\dz P^{\rm NL}+\eps F_{\rm v},\\
\dsp \dt \widetilde{\rho}+\widetilde{w}\frac{d}{dz}\rho_{\rm eq}=\eps f,
\end{array}\right.
$$
with $f=-\eps U\cdot \grad \rho$. Proceeding as for Lemma \ref{lemmpress}, one gets
$$
\Vert \grad \widetilde{P} \Vert_{H^{\nu-1}}\lesssim \Vert \widetilde{\rho}\Vert_{H^{\nu-1}}
\quad\mbox{ and }\quad
\Vert \grad {P}^{\rm NL} \Vert_{H^{\nu-1}}\leq C\big( \Vert(U, \rho)\Vert_{H^\nu}\big)
$$
while $F$ and $f$ are also bounded from above in $H^{\nu-1}(\cS)$ by $ C\big( \Vert(U, \rho)\Vert_{H^\nu}\big)$. It follows as in the proof of Theorem \ref{theorem1} that
$$
\frac{d}{dt}\Big( \Vert \widetilde{U}\Vert_{H^{\nu-1}}^2+\Vert \widetilde{\rho}\Vert^2_{H^{\nu-1}}\Big) \leq C \Vert \widetilde{\rho}\Vert_{H^{\nu-1}}^2 + \eps C\big( \Vert(U, \rho)\Vert_{H^\nu}\big)\Big( \Vert \widetilde{U}\Vert_{H^{\nu-1}}+\Vert \widetilde{\rho}\Vert_{H^{\nu-1}}\Big),
$$
and the estimate provided in the proposition classically follows using a Gronwall-type argument.
\end{proof}
\subsection{Modal decomposition}\label{SectSL}

Simple manipulations of the linear equations  \eqref{Euler_pert_lin}-\eqref{BC_pert_lin} show that the vertical velocity $w$ must satisfy
\begin{equation}\label{eqbig}
\dt^2\Big[\Big(\Delta +\frac{1}{\rho_{\rm eq}}\dz \big(\rho_{\rm eq}\dz
\cdot \big)\Big)w\Big]+N^2 \Delta w=0,
\end{equation}
where $N=N(z)$ is the Br\"unt-Vais\"al\"a frequency,
$$
N^2=-\frac{\rho'_{\rm eq}}{\rho_{\rm eq}}g.
$$

%\begin{remark}
%In the physical literature, the derivative of the equilibrium density
%$\rhoeq$ is neglected in the first term of \eqref{eqbig} -- this is
%called the Boussinesq approximation -- leading to the simplified
%equation
%$$
%\dt^2\Big[\Big(\Delta +\dz^2\Big)w\Big]+N^2 \Delta w=0;
%$$
%we do not need to make this approximation in our analysis.
%\end{remark}

In order to solve \eqref{eqbig}, it is therefore quite natural to  decompose $w$ on the orthonormal 
basis $({\bf f}_n)_{n\in\N}$ of $L^2([-1,0],\rho_{\rm eq}N^2dz)$ formed by the eigenfunctions of the following
Sturm-Liouville problem,
\begin{equation}\label{SL}
\left\lbrace
\begin{array}{l}
\dsp \frac{d}{dz}\big(\rho_{\rm eq}\frac{d}{dz}{\bf
  f}_n\big)+\frac{\rho_{\rm eq}N^2}{c_n^2}{\bf f}_n=0,\\
\dsp  {\bf f}_n(-H)={\bf f}_n(0)=0,
\end{array}\right.
\end{equation}
with corresponding eigenvalues $\dsp (\frac{1}{c_n^2})_{n\in\N}$ satisfying $c_1>c_2>\dots>c_n>\dots$ and
 $c_j>0$ for all $j\in\N^*$; we also have the asymptotic behavior $c_n=O(1/n)$ as $n\to \infty$.
\begin{example}\label{ex1}
We show in Figure \ref{fig:dummy} a typical stratification for an ocean of depth $H=3500{\rm m}$ (see \cite{Chelton}) and the corresponding Brunt-Vais\"al\"a frequency.
   \begin{figure}[!ht]
     \subfloat[Density (${\rm kg}.{\rm m}^{-3}$) \label{subfig-1:dummy}]{%
       \includegraphics[width=0.45\textwidth]{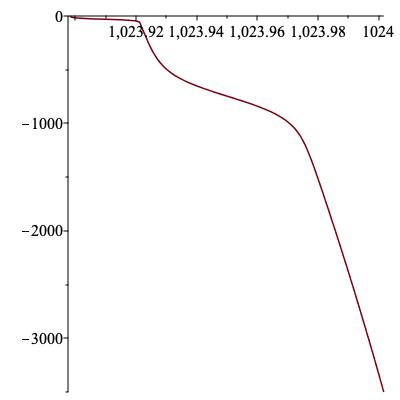}
     }
     \hfill
     \subfloat[$N$ (cycles/hour)\label{subfig-2:dummy}]{%
       \includegraphics[width=0.45\textwidth]{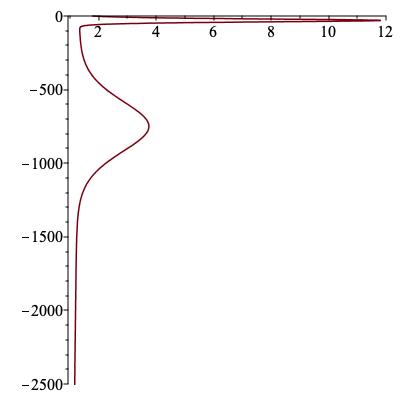}
     }
     \caption{Example of the vertical dependence of the density and Brunt-Vais\"al\"a frequency in an ocean}
     \label{fig:dummy}
   \end{figure}
The speeds associated to the first modes in the Sturm-Liouville problem \eqref{SL} are represented in Figure \ref{subfig2a}.
   \begin{figure}[!ht]
     \subfloat[For the ocean stratification of Fig. \ref{subfig-1:dummy} \label{subfig2a}]{%
       \includegraphics[width=0.49\textwidth]{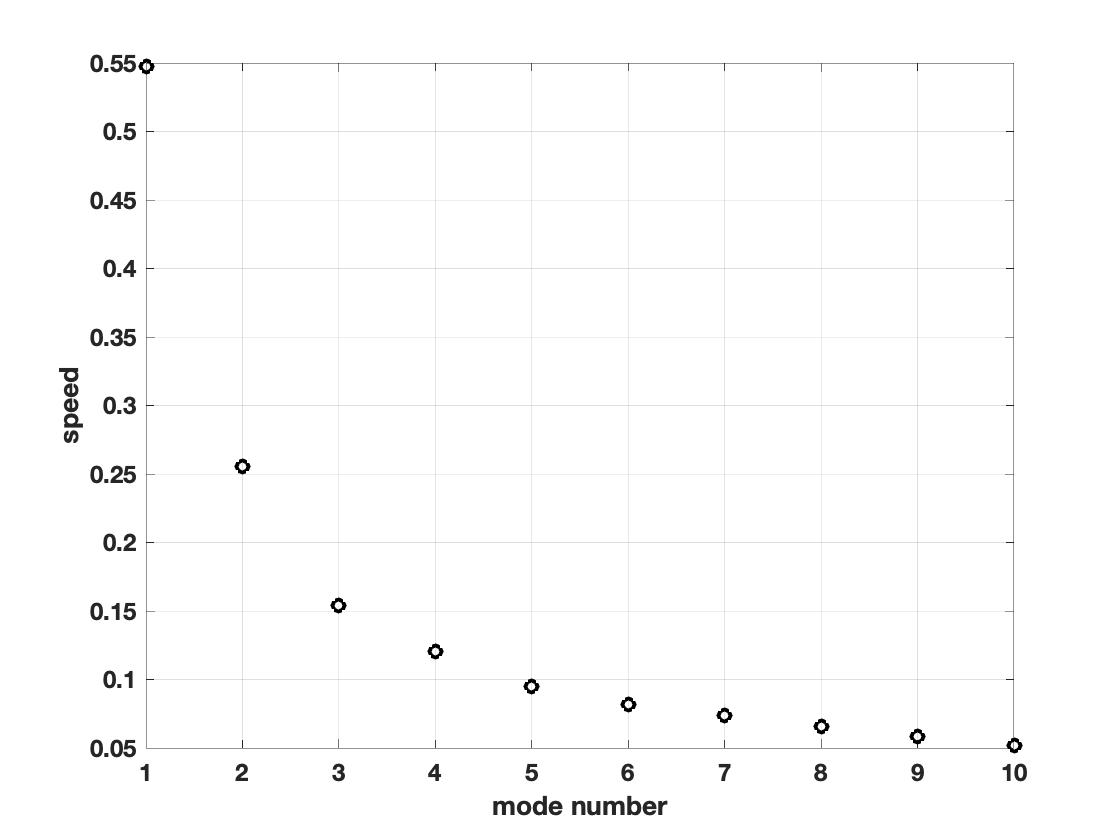}
     }
     \hfill
     \subfloat[For the sharp stratification \eqref{densitysharp} \label{subfig2b}]{%
       \includegraphics[width=0.49\textwidth]{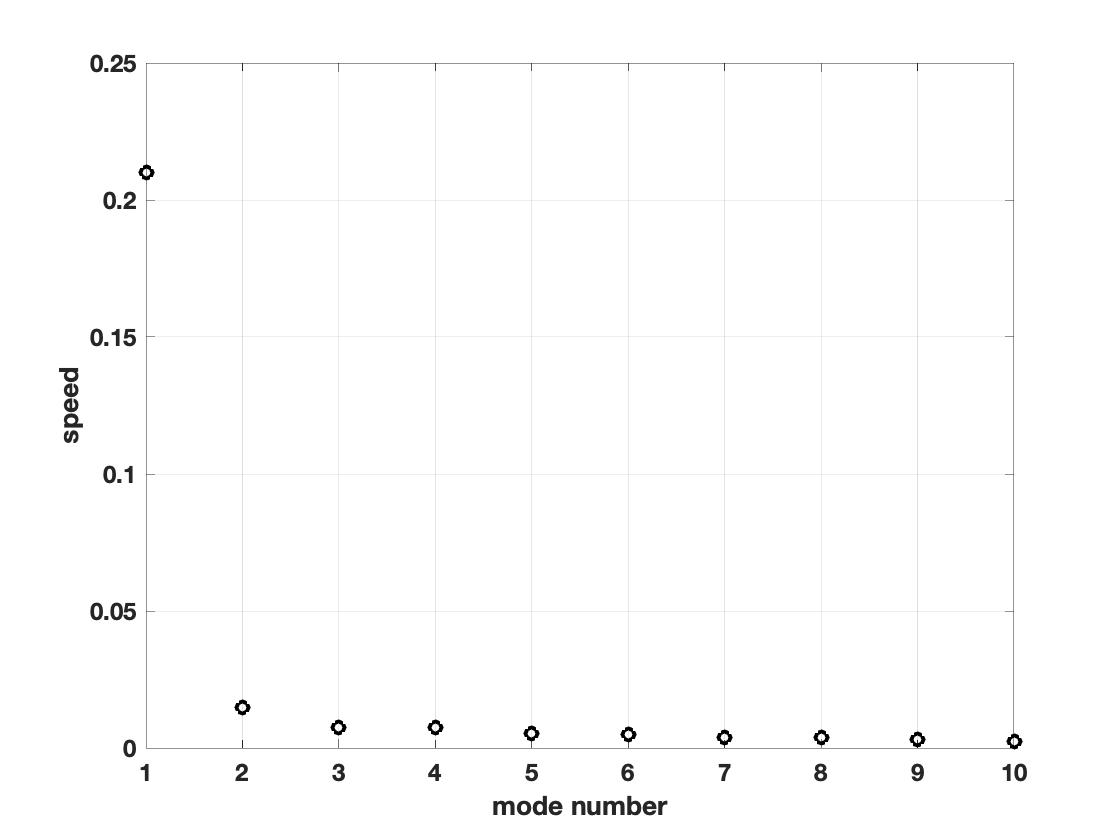}
     }
     \caption{The speeds associated to the first modes}
     \label{fig2}
   \end{figure}

%\begin{figure}
%\includegraphics[width=0.7\textwidth]{Figs/speed}
%\caption{The speeds associated to the first modes}
%\label{fig2}
%\end{figure}
Finally, the first six eigenfunctions of the Sturm-Liouville basis $({\bf f}_n)$ associated to \eqref{SL} are represented in Figure \ref{fig3}.
\begin{figure}
\includegraphics[width=0.9\textwidth]{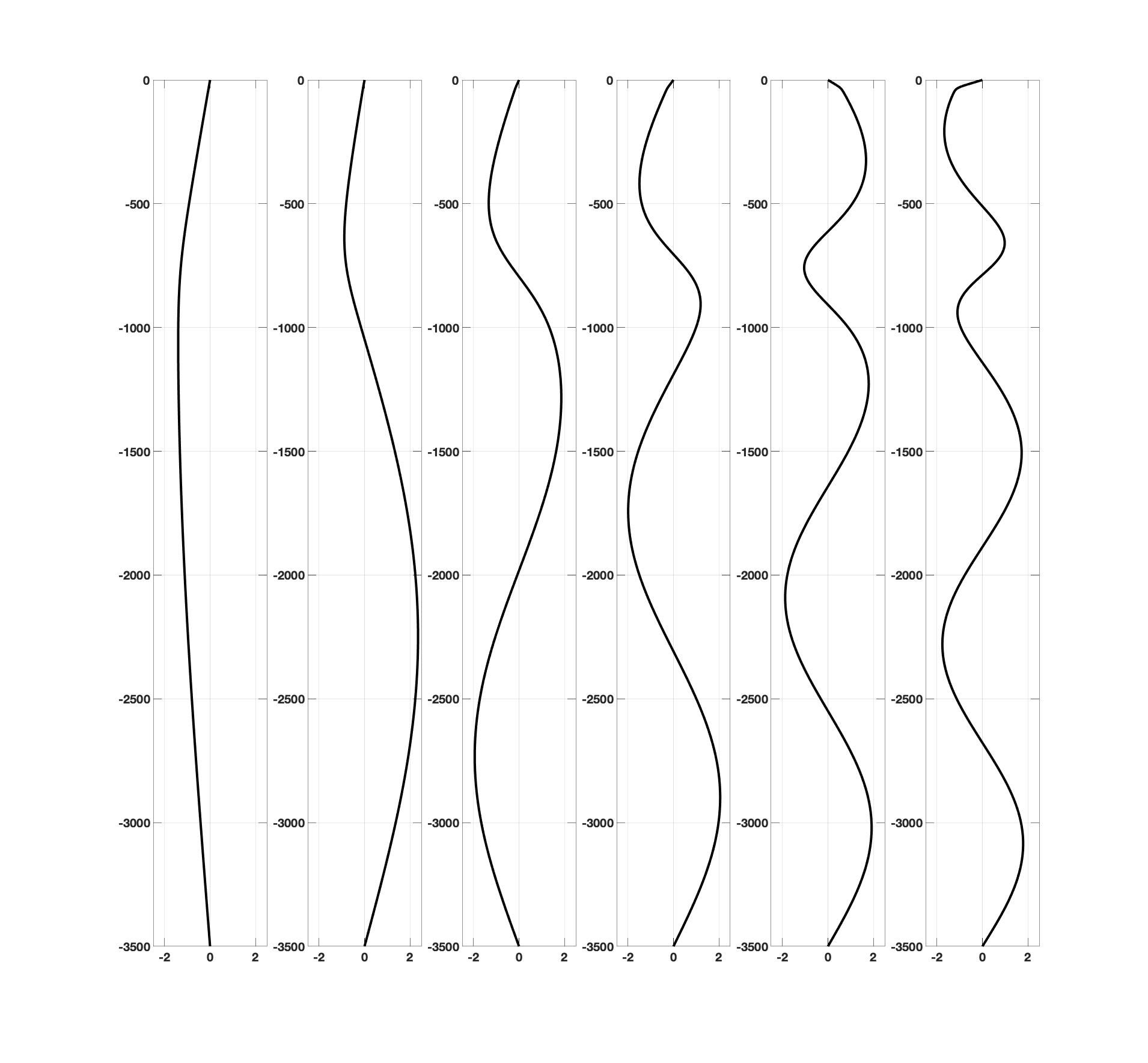}
\caption{The first $6$ vertical modes of the Sturm-Liouville decomposition \eqref{SL} (from left to right)}
\label{fig3}
\end{figure}
\end{example}
We therefore look for $w$ under the form
\begin{equation}\label{decompw}
w(t,X,z)=\sum_{n=1}^\infty w_n(t,X){\bf f}_n(z).
\end{equation}
Since the coefficients of the system \eqref{Euler_pert_lin} depend on
$z$, the velocity, pressure and density fields must be decomposed on
other orthonormal basis that are introduced in the following lemma.

\begin{lemma}\label{lemmebases}
Let $\rho_{\rm eq}\in W^{1,\infty}(-H,0)$ and $\alpha=\big(\int_{-H}^0\rhoeq^{-1}\big)^{-1/2}$,
assume that
$$
\inf_{z\in(-H,0)}\rho_{\rm eq}(z)>0 \quad\mbox{ and }\quad \inf_{z\in(-H,0)}N(z)>0,
$$
 and let 
$({\bf f}_n)_{n\geq 1}$ be the orthonormal basis of $L^2((-H,0),\rho_{\rm
  eq}N^2 dz)$ formed by the eigenfunctions of \eqref{SL}. Then\\
{\bf i.} The sequence $(\rho_{\rm eq}N^2{\bf f}_n)_{n\geq 1}$ forms an
orthonormal basis of the weighted space $L^2((-H,0),(\rho_{\rm eq}N^2)^{-1}dz)$.\\
{\bf ii.} The sequence $({\bf g}_n)_{n\geq 0}$ with ${\bf
  g}_0=\alpha\rho_{\rm eq}^{-1}$ (with $\alpha$ as above), and ${\bf
  g}_n=c_n{\bf f}_n'$ for $n\geq 1$ forms an orthonormal basis of
$L^2((-H,0),\rhoeq dz)$.\\
{\bf iii.} The sequence $(\rho_{\rm eq}{\bf g}_n)_{n\geq 0}$ forms an orthonormal basis of
$L^2((-H,0),\rhoeq^{-1} dz)$.
\end{lemma}
\begin{proof}
The first point is straightforward. \\
For the second point, remark first
that integrating by parts and using \eqref{SL}, one gets
\begin{align*}
\forall m,n\geq 1,\qquad ({\bf g}_m,{\bf
  g}_n)_{L^2_{\rhoeq}}&=c_mc_n(\rhoeq{\bf f}'_m, {\bf f}'_n)\\
&=-c_mc_n((\rhoeq{\bf f}'_m)', {\bf f}_n)\\
&=\frac{c_n}{c_m}({\bf f}_m,{\bf f}_n)_{L^2_{\rhoeq N^2}}\\
&=\delta_{mn},
\end{align*}
which proves that $({\bf g}_n)_{n\geq 1}$ is an orthonormal family for
$L^2((-H,0),\rhoeq dz)$. We
therefore need to prove that it forms an orthonormal basis if we
complement it with ${\bf g}_0$. The scalar $\alpha$ is chosen so that
${\bf g}_0$ is a unit vector. In order to check that it is orthogonal
to ${\bf g}_n$ for $n\geq 1$, just remark that
\begin{align*}
({\bf g}_0,{\bf g}_n) _{L^2_{\rhoeq}}&=\alpha c_n \int_{-H}^0{\bf
  f}'_n\\
&=0
\end{align*}
since ${\bf f}_n$ vanishes at the boundaries. The family $({\bf
  g}_n)_{n\geq 0}$ is therefore orthonormal and we just have to prove
that its orthogonal reduces to $\{0\}$. Let therefore
$f\in L^2((-H,0),\rhoeq dz)$
be such that
$$
\forall n\in \N,\qquad (f,{\bf g}_n)_{L^2_{\rhoeq}}=0
$$
and let us prove that $f=0$. From the case $n=0$, one deduces that
$\int_{-H}^0 f=0$. Denoting $F(z)=\int_{-H}^z f$, one also deduces
that  for all
$n\geq 1$ one has
\begin{align*}
0&= c_n (f,{\bf f}'_n)_{L^2_{\rhoeq}}\\
&=-c_n (F,(\rhoeq{\bf f}'_n)')\\
&=c_n^{-1}(F,{\bf f}_n)_{L^2_{\rhoeq N^2}},
\end{align*}
where we used the fact that $F$ vanishes at the boundaries (owing to
the fact that $f$ has zero mean) to derive
the second equality. Since $({\bf f}_n)_{n\geq 1}$ is an orthonormal
basis of $L^2((-H,0),\rho_{\rm
  eq}N^2 dz)$, we deduce that $F=0$ and therefore that $f=0$.\\
The third point is then directly deduced from the second one.
\end{proof}

Owing to the lemma, we  can  look for $V$, $P$ and $\rho$
under the following expansions
\begin{equation}\label{decomp}
\begin{array}{lcl}
V(t,X,z)&=&\dsp \sum_{n=0}^\infty V_n (t,X)
{\bf g}_n(z),\\
P(t,X,z)&=&\dsp \sum_{n=0}^\infty P_n(t,X)\rho_{\rm eq}(z){\bf g}_n(z),\\
\rho(t,X,z)&=&\dsp \sum_{n=1}^\infty \frac{1}{g}\rho_n
(t,X)\rho_{\rm eq}(z)N^2(z){\bf f}_n(z).
\end{array}
\end{equation}
\begin{remark}\label{remnormes}
Note in particular that
$$
\Abs{V}^2_{L^2_{\rhoeq}}=\sum_{n=0}^\infty \abs{V_n}_2^2,\quad
\Abs{w}^2_{L^2_{\rhoeq N^2 }}=\sum_{n=0}^\infty \abs{w_n}_2^2,\quad
\Abs{g\rho}^2_{L^2_{\frac{1}{\rhoeq N^2}}}=\sum_{n=0}^\infty \abs{\rho_n}_2^2;
$$
more generally, one has for all $s\in \R$,
$$
\Abs{V}^2_{H^{s,0}_{\rhoeq}}=\sum_{n=0}^\infty \abs{V_n}_{H^s}^2,\quad
\Abs{w}^2_{H^{s,0}_{\rhoeq N^2 }}=\sum_{n=0}^\infty \abs{w_n}_{H^s}^2,\quad
\Abs{g\rho}^2_{H^{s,0}_{\frac{1}{\rhoeq N^2}}}=\sum_{n=0}^\infty \abs{\rho_n}_{H^s}^2.
$$
\end{remark}

\subsection{The case of constant $N$: no dispersive mixing}\label{sectNconst}

The following proposition restates the system of equations
\eqref{Euler_pert_lin} as a set of equations on the coefficients of
these decompositions. We first assume that the Br\"unt-Vais\"al\"a frequency $N$ is constant.
\begin{proposition}\label{proplin}
Let $\rho_{\rm eq}\in W^{1,\infty}(-H,0)$ be such that $N^2=-\frac{\rho'_{\rm eq}}{\rho_{\rm eq}}$ is constant and strictly positive on $(-H,0)$.
Let also $s\geq 0$ and $T>0$ and $(U^{\rm lin}, \rho^{\rm lin})\in C([0,T];H^{s,0})$ solve the linear equations  \eqref{Euler_pert_lin}-\eqref{BC_pert_lin} and assume that the vertical component of the vorticity is zero, that is, $\nabla^\perp\cdot V^{\rm lin}_{\vert_{t=0}}=0$.\\
The coefficients $(V_n,w_n,\rho_n,P_n)$ provided by the
decompositions \eqref{decompw} and \eqref{decomp} are found
by solving, for all $n\geq 1$, the equations
$$
\begin{cases}
\dsp \big(1-\frac{c_n^2}{N^2}\Delta\big)\dt V_n+c_n \nabla\rho_n=0,\\
\dsp \dt \rho_n+c_n\nabla\cdot V_n=0,
\end{cases}
$$
with initial data given by the coefficients of the decomposition of and $V^0$ and $\rho^0$, and $w_n$ and $P_n$ are given by
$$
w_n=-c_n \nabla\cdot V_n, \quad \mbox{ and }\quad P_n=c_n \rho_n+\frac{c_n}{N^2}\dt w_n,
$$
while for $n=0$, one must have $P_0=0$ and $ V_0=0$. Moreover,
$$
\sum_{n=1}^\infty\big[ \vert \rho_n\vert_{H^s}^2+ \vert V_n\vert^2_{H^s}+ \vert w_n\vert_{H^{s}}^2 \big] <\infty.
$$
\end{proposition}
\begin{remark}
The assumption that $\nabla^\perp\cdot V^{\rm lin}_{\vert_{t=0}}=0$ is not restrictive. Indeed, one deduces from the first equation of \eqref{Euler_pert_lin} that $\nabla^\perp\cdot V^{\rm lin}$ is constant in time. We can therefore always decompose $V^{\rm in}$ as  $V^{\rm in}(t,X,z)=V_{\rm I}(X,z)+V_{\rm II}(t,X,z)$ with $\nabla\cdot V_{\rm I}=\nabla^\perp\cdot V_{\rm II}=0$ and apply the proposition to $V_{\rm II}(t,X,z)$.
\end{remark}
\begin{remark}
The system solved by the $(V_n,\rho_n)$ is a linear Boussinesq-type system similar to those arising in the study of shallow water waves (e.g. \cite{Lannes_book}). It is striking that for water waves, the dispersive perturbation appears as an approximation of a nonlocal operator (related to a Dirichlet-Neumann operator) which is valid only in shallow water. Here, the simple, differential, form of the dispersive term is valid outside the shallow water regime (see Section \ref{sectSW} for an investigation of this latter).
\end{remark}
\begin{proof}
The proof consists in decomposing the equations of \eqref{Euler_pert_lin} on several basis related to the Sturm-Liouville basis $({\bf f}_n)_{1\leq n}$. We have therefore to decompose time and space derivatives of $V$, $w$, $\rho$ and $P$ on such basis. This is straightforward for time and horizontal derivatives; for instance,
$$
\dt V =\sum_{n=0}^\infty (\dt V_n ) {\bf g}_n,\qquad \nabla P=\sum_{n=0}^\infty (\nabla P_n) \rho_{\rm eq} {\bf g}_n,
$$
but this is less obvious for vertical derivatives, which are considered in the following lemma.
\begin{lemma}\label{lemder}
Let $s\in \R$.\\
{\bf i.}  The following two assertions are equivalent,
\begin{enumerate}
\item  One has $w\in H^{s,1}(\cS)$ and $w_{\vert_{z=-H,0}}=0$.
\item One has $w=\sum_{n=1}^\infty w_n {\bf f}_n$ and $\sum_{n=1}^\infty \abs{w_n}_{H^s}^2+ \sum_{n=1}^\infty c_n^{-2}\abs{w_n}_{H^{s-1}}^2 <\infty$.
\end{enumerate}
Moreover, one has
$$
\dz w=\sum_{n=1}^\infty \frac{1}{c_n}w_n {\bf g}_n.
$$
{\bf ii.}  The following two assertions are equivalent,
\begin{enumerate}
\item  One has $\nabla P\in H^{s,0}$ and $\dz P\in H^{s+1,0}$.
\item One has $P=\sum_{n=0}^\infty P_n {\bf \rho}_{\rm eq}{\bf g}_n$ and $ \sum_{n=1}^\infty c_n^{-2}\abs{P_n}_{H^{s+1}}^2 <\infty$.
\end{enumerate}
Moreover, one has
$$
\dz P=-\sum_{n=1}^\infty \frac{1}{c_n}P_n \rho_{\rm eq} N^2 {\bf f}_n.
$$

\end{lemma}
\begin{proof}[Proof of the lemma]
{\bf i.} Let us first prove the direct implication. By assumption, $\dz w\in L^2_z H^{s-1}(\R^d)$ and can therefore be decomposed on the basis $({\bf g}_n)_{0\leq n}$, which, we recall, is orthonormal for the $L^2_{\rho_{\rm eq}}$ scalar product.
$$
\dz w = \sum_{n=0}^\infty (\dz w, {\bf g}_n)_{L^2_{\rho_{\rm eq}}} {\bf g}_n.
$$
Since $w$ vanishes at the boundaries, we have, 
\begin{align*}
(\dz w, {\bf g}_n)_{L^2_{\rho_{\rm eq}}} &=-\int_{-H}^0 w c_n \frac{d}{dz}\big( \rho_{\rm eq}\frac{d}{dz} {\bf f}_n \big)\\
&=\frac{1}{c_n}w_n,
\end{align*}
using the definition of ${\bf f}_n$ and ${\bf g}_n$ (when $n=0$ the last expression should be replaced by $0$). The result follows therefore from Parseval identity.\\
For the reverse inequality, Parseval identity directly yields $w\in H^{s,1}(\cS)$, but it remains to show that the trace vanishes at the boundary. This is the case because $w$ is an infinite sum of functions that vanish at the boundary and that, under our assumptions, and recalling that the eigenfunctions are uniformly bounded \cite{Fulton}, this sum is absolutely convergent.\\
{\bf ii.} Since $\dz P \in H^{s+1,0}$ and $\nabla P\in H^{s,0}$, there are functions $q_n\in H^{s+1}(\R^d)$ and $Q_n\in H^s(\R^d)^d$ such that
$$
\dz P=\sum_{n=1}^\infty q_n \rho_{\rm eq} N^2 {\bf f}_n \quad\mbox{ and }\quad \nabla P= \sum_{n=0}^\infty Q_n \rho_{\rm eq} {\bf g}_n.
$$
Since $\nabla\dz P \in H^{s,0}$, we can decompose it  on the basis $(\rho_{\rm eq}N^2 {\bf f}_n)$ which is orthonormal for the $L^2((-H,0),(\rho_{\rm eq}N^2)^{-1}{\rm d}z)$ scalar product, and we have
$$
(\dz \nabla P, \rho_{\rm eq}N^2 {\bf f}_n)_{L^2_{(\rho_{\rm eq}N^2)^{-1}}}=-(\nabla P,\frac{1}{c_n}{\bf g}_n)_{L^2}=-\frac{1}{c_n}Q_n
$$
(note that contrary to the first point, it is not required that $P$ vanishes at the boundary, because the eigenfunctions ${\bf f}_n$ satisfy such a cancellation). This yields that $Q_n=-c_n\nabla q_n$ and the result follows by setting $P_n=-c_n q_n$.
\end{proof}
Using Lemmas \ref{lemmebases} and \ref{lemder}, the system of  equations
\eqref{Euler_pert_lin}  is equivalent to the relations obtained by
taking the $L^2((-H,0),dz)$ scalar product of these equations with
$\rhoeq{\bf g}_n$ ($n\geq0$), $\rhoeq  {\bf f}_n$ ($n\geq 1$),
${\bf f}_n$ ($n\geq 1$) and $ \rhoeq {\bf g}_n$ ($n\geq
0$) respectively. These relations are
\begin{equation}
\label{Montp0}
\left\lbrace
\begin{array}{l}
\dsp \dt V_0+\nabla P_0=0,\\
\dsp \nabla\cdot V_0=0
\end{array}\right.
\end{equation}
and
\begin{equation}
\label{Montp}
\forall n\geq 1,\qquad \left\lbrace
\begin{array}{l}
\dsp \dt V_n+\nabla P_n=0 , \\
\dsp \frac{1}{N^2} \dt w_n+\rho_n-\frac{P_n}{c_n}=0,\\
\dsp \dt \rho_n- w_n=0,\\
\dsp  \frac{1}{c_n}w_n =-\nabla\cdot V_n,
\end{array}\right.
\end{equation}
where the assumption that $N$ is independent of $z$ has been used to
derive the second equation in \eqref{Montp}. The equations on the coefficients given in the proposition follow easily. Finally, the convergence of the summation given in the last point of the proposition is a consequence of Remark \ref{remnormes}.
\end{proof}
\begin{remark}\label{remmix}
The second equation in \eqref{Montp} is actually
$$
 \frac{P_n}{c_n}-\rho_n=(\dt w,{\bf f}_n)_{L^2_{\rho_{\rm eq}}},
$$
and mixes all the modes if $N^2$ is not constant. Since we assumed that $N^2$ is constant in the statement of the proposition, we have $(\dt w,{\bf
  f}_n)_{L^2_{\rho_{\rm eq}}}=\frac{1}{N^2} \dt w_n$. The general case is considered in Proposition \ref{proplinbis} below.
\end{remark}
\begin{remark}
One directly deduces from \eqref{Montp} that
$$
\frac{d}{dt}\big( \abs{V_n}_2^2+\frac{c_n^2}{N^2}\abs{\nabla\cdot
  V_n}_2^2+\abs{\rho_n}_2^2\big)=
\frac{d}{dt}\big( \abs{V_n}_2^2+\frac{1}{N^2}\abs{w_n}_2^2+\abs{\rho_n}_2^2\big) =
0.
$$
Owing to Remark \ref{remnormes}, we get after summing over $n\in \N$,
$$
\frac{d}{dt}\int_{\R^d}\Big(\frac{1}{2} \rhoeq\abs{U}^2+
\frac{1}{2}\rho^2
\frac{g^2}{N^2\rhoeq}\Big)=0,
$$
which is the energy identity satisfied by the solution to the linear equations  \eqref{Euler_pert_lin}-\eqref{BC_pert_lin}.
\end{remark}
\begin{remark}\label{remIHP}
In the physical literature, the dynamics of the small amplitude
adjustments is generally described in terms of the vertical velocity
$w$ through \eqref{eqbig}, or equivalently, after decomposing on the
basis $({\bf f}_n)_{n\geq 1}$, by
$$
\forall n\in \N^*, \qquad \dt ^2 w_n-c_n^2 (1-\frac{c_n^2}{N^2}\Delta)^{-1}\Delta w_n=0.
$$
This equation can of course be easily deduced from \eqref{Montp}. This configuration will be considered in Section \ref{sectSW} below where the more complex nonlinear case is considered.
\end{remark}

The representation of the solution to the linear equations \eqref{Euler_pert_lin}-\eqref{BC_pert_lin} based on the modal decomposition  \eqref{decompw} and \eqref{decomp} brings some interesting qualitative insight on the behavior of the waves. It should be used with caution however; indeed, these decompositions involve infinite sums that without further assumptions converge slowly, and in general no better than in the $L^2$-sense. For instance, the modal decomposition for the density is 
$$
\rho(t,X,z)=\sum_{n=1}^\infty \frac{1}{g} \rho_n(t,X)\rho_{\rm eq}(z)N^2{\bf f}_n(z);
$$
all the terms in the summation vanish at the boundary so that, when $\rho$ does not vanish at the boundary, this summation \emph{cannot} converge uniformly in, say, $L^\infty([-H,0]; H^s(\R^d))$. Representing numerically the solution by a finite sum of modes with coefficients computed as in Proposition \ref{proplin} is therefore inaccurate. This is due to the lack of regularity with respect to $z$ of the solution considered in Proposition \ref{proplin}. Indeed, as it appears in the Sturm-Liouville problem \eqref{SL}, differentiation with respect to $z$ is roughly equivalent to the multiplication of the coefficients of the modal expansion by a factor of size $O(1/c_n)=O(n)$. If we consider a function $f \in H^m(-H,0)$, represented  under the form
$$
F=\sum_{n=1}^\infty F_n {\bf f}_n,
$$
it is therefore expected that $\sum  n^{2m}\vert F_n\vert^2 <\infty$. As seen above, this is however false without additional assumptions on the behavior of $F$ at the boundaries because, one cannot identify $\dz^m F$ with $\sum_{n=1}^\infty F_n \frac{d^m}{dz^m}{\bf f}_n$.\\
The following proposition provides such additional conditions on the initial data under which the modal decomposition is more strongly (and in particular uniformly) convergent. To state these conditions, we
need to define the second order differential operators $T_1$, $T_2$  as
$$
T_1=\dz\Big( \frac{1}{-\rho'_{\rm eq}}\dz \big( \rho_{\rm eq}\cdot \big) \Big) \quad \mbox{ and } \quad 
T_2=\dz\Big(\rho_{\rm eq} \dz \big( \frac{1}{-\rho'_{\rm eq}} \cdot \big) \Big);
$$
in the statement below, the condition on $V^0$ must be removed in the case $l=0$ and we denote with a star the adjoint operator with the standard $L^2((-H,0),{\rm d}z)$ scalar product.
\begin{proposition}\label{proplinstrong}
Under the assumptions of Proposition \ref{proplin}, assume moreover that for some $k\in \N^*$ and $\nu=2k$ or $2k+1$, one has $\rho_{\rm eq}\in W^{\nu,\infty}$ and that the solution $(U^{\rm lin},\rho^{\rm lin})$ to the linear equation belongs to $C([0,T], H^{s,\nu}(\cS))$, and also that its initial data $(U^0,\rho^0)$ satisfy the additional conditions
$$
\dz (T_1)^{l-1}V^0=0,\qquad \big(T_2^*\big)^l w^0=0,\quad (T_2)^{l}\rho^0 =0,
$$
for all $0\leq l\leq k$. Then the following convergence holds
$$
\sum_{n=1}^\infty\big[ \frac{1}{c_n^{2\nu}}\vert \rho_n\vert_{H^{s-\nu}}^2+ \frac{1}{c_n^{2\nu}}\vert V_n\vert^2_{H^{s-\nu}}+\frac{1}{c_n^{2\nu}} \vert w_n\vert_{H^{s-\nu}}^2 \big] <\infty.
$$
\end{proposition}
\begin{proof}
The following key lemma shows the importance of the differential operators $T_1$ and $T_2$ introduced above.
\begin{lemma}\label{keylemma}
Assume that $N^2=g\frac{-\rho'_{\rm eq}}{\rho_{\rm eq}}$ does not depend on $z$ and that $(U,\rho)$ be a smooth enough solution of the linear system
\begin{equation}\label{lemlineq}
\begin{cases}
\dt V + \frac{1}{\rho_{\rm eq}}\nabla   P&= 0\\
\dt w + \frac{1}{\rho_{\rm eq}}(\dz P+\rho g \big)&= 0\\
\dt  \rho +\rho'_{\rm eq}  w&=0,\\
\nabla \cdot V+\dz w&=0.
\end{cases}
\end{equation}
Defining $V^\sharp$, $w^\sharp$, $\rho^\sharp$ and $P^\sharp$ through
$$
V^\sharp=T_1 V, \qquad w^\sharp=T_2^* w, \qquad \rho^\sharp=T_2 \rho, \qquad P^\sharp= T_1^* P,
$$
one has
$$
\begin{cases}
\dt V^\sharp + \frac{1}{\rho_{\rm eq}}\nabla   P^\sharp&= 0\\
\dt w^\sharp + \frac{1}{\rho_{\rm eq}}(\dz P^\sharp+\rho^\sharp g \big)&= 0\\
\dt  \rho^\sharp +\rho'_{\rm eq}  w^\sharp&=0,\\
\nabla \cdot V^\sharp+\dz w^\sharp&=0.
\end{cases}
$$
\end{lemma}
\begin{proof}[Proof of the lemma]
The first, third and fourth equations follow from simple computations by applying $T_1$ to the first equation of \eqref{lemlineq}, $T_2$ to the third one and $T_1$ to the fourth one. Let us now define
$$
T_3=\frac{1}{\rho_{\rm eq}}\Big( \rho_{\rm eq}\dz \big( \frac{\rho_{\rm eq}}{-\rho'_{\rm eq}}\cdot \big) \Big);
$$
applying $T_3$ to the second equation in \eqref{lemlineq}, one gets
$$
T_3\dt  w + \frac{1}{\rho_{\rm eq}}(\dz P^\sharp+\rho^\sharp g \big)= 0\\
$$
and the result follows from the observation that if $N^2$ does not depend on $z$, then $T_3=T_2^*$.
\end{proof}
It is then quite easy to check that the conditions made in the statement of the proposition on the initial data are propagated by the equations. And since the system satisfied by $(U^\sharp,\rho^\sharp, P^\sharp)$ has exactly the same structure as the original one, it is enough to prove that result for $\nu=1$ and $\nu=2$.  The case $\nu=1$ follows from Lemma \ref{lemder}  for $w$ and $P$. For $V$, the results is obtained as for $P$; finally, the assumption that $\rho$ vanishes at the boundaries is propagated from the initial condition, and the result for $\rho$ can be obtained as for $w$.\\
We now prove the case $\nu=2$. Since under the assumptions of the proposition one has $T_2^* w \in L^2([-H,0];H^{s-2}(\R^d))$, we can represent it on the basis ${\bf f}_n$ with coefficients
$$
\big( T_2^* w, {\bf f}_n \big)_{L^2_{N^2 \rho_eq}}=-\frac{1}{c_n}\big( \dz w, {\bf g}_n \big)_{L^2_{\rho_{\rm eq}}}=-\frac{1}{c_n^2} w_n,
$$
the last equality stemming from Lemma \ref{lemder}. Since moreover $w_n=-c_n\nabla \cdot V_n$ (and that $\nabla^\perp \cdot V_n=0)$, we then deduce that $\sum_{n=1}^\infty \frac{1}{c_n^{4-2}}\abs{ V_n }_{H^{s-2k+1}} <\infty$. \\
Since $\rho$ is now assumed to vanish at the boundary, there is a representation formula for $\dz \rho$ similar to the one provided for $w$ in Lemma \ref{lemder} and one can conclude as above. For $V$, there is a representation formula similar to the one given for $P$ in Lemma \ref{lemder} and the fact that $\dz V$ vanishes at the boundary provides a representation formula for $T_1 V$ that provides the result through Parseval identity.
\end{proof}

\subsection{The case of variable $N$ and dispersive mixing}\label{sectNpasconst}

We have already noted in Remark \ref{remmix} that it is necessary  that $N^2$ be constant in the derivation of the equations given in Proposition \ref{proplin} for the coefficients of the modal decomposition. The key point was that when  $N^2$ is constant then one can write
\begin{align*}
 ({\bf f}_m,{\bf f}_n)_{L^2_{\rho_{\rm eq}}}&=\frac{1}{N^2} ({\bf
   f}_m,{\bf f}_n)_{L^2_{\rho_{\rm eq}N^2}},\\
&=\frac{1}{N^2}\delta_{mn}
\end{align*}
where $\delta_{mn}=1$ if $m=n$ and $0$ otherwise.
In the general case where
$N^2$ depends on $z$, this is no longer the case and we are therefore led to define the interaction coefficients $\alpha_{mn}$ as
\begin{equation}\label{interaction}
\alpha_{mn}=(\bff_m,\bff_n)_{L^2_{\rho_{\rm eq}}},
\end{equation}
and the second equation in \eqref{Montp} becomes
\begin{align*}
 \frac{P_n}{c_n}-\rho_n&=(\dt w,{\bf f}_n)_{L^2_{\rho_{\rm eq}}},\\
 &=\sum_{m=1}^\infty  \alpha_{mn}\dt w_m\\
 &=-\sum_{m=1}^\infty c_m \alpha_{mn}\nabla\cdot \dt V_m;
\end{align*}
we therefore have the following generalization of Proposition \ref{proplin}.
\begin{proposition}\label{proplinbis}
Let $\rho_{\rm eq}\in W^{1,\infty}(-H,0)$ be such that $N^2=-\frac{\rho'_{\rm eq}}{\rho_{\rm eq}}$ is a strictly positive function on $(-H,0)$.
Let also $s\geq 0$ and $T>0$ and $(U^{\rm lin}, \rho^{\rm lin})\in C([0,T];H^{s,0})$ solve the linear equations  \eqref{Euler_pert_lin}-\eqref{BC_pert_lin} and assume that the vertical component of the vorticity is zero, that is, $\nabla^\perp\cdot V^{\rm lin}=0$.\\
The coefficients $(V_n,w_n,\rho_n,P_n)$ provided by the
decompositions \eqref{decompw} and \eqref{decomp} are found
by solving, for all $n\geq 1$, the equations
$$
\begin{cases}
\dt V_n- \dsp \sum_{m=1}^\infty c_mc_n \alpha_{mn}\Delta\dt V_m+c_n \nabla\rho_n=0,\\
\dsp \dt \rho_n+c_n\nabla\cdot V_n=0
\end{cases}
$$
with initial data given by the coefficients of the decomposition of and $V^0$ and $\rho^0$, and $w_n$ and $P_n$ are given by
$$
w_n=-c_n \nabla\cdot V_n, \quad \mbox{ and }\quad P_n=c_n \rho_n+ c_n \sum_{m=1}^\infty \alpha_{mn} \dt w_m ,
$$
while for $n=0$, one must have $P_0=0$ and $ V_0=0$. 
\end{proposition}
\begin{remark}
As in Proposition \ref{proplin}, one gets that the sequences $(\abs{V_n}_{H^s})_n$ and $(\abs{\rho_n}_{H^s})_n$ are in $l^2(\N)$, but it is not possible to improve this convergence as in Proposition \ref{proplinstrong}. Indeed, the proof of this proposition relied on Lemma \ref{keylemma} which requires that $N^2$ is constant. The best one can get is to extend to non constant $N$ the case $\nu=1$ of Proposition \ref{proplinstrong}; this ensures that if $(U^{\rm lin}, \rho^{\rm lin})\in C([0,T];H^{s,1})$,  then one has convergence  of $(c_n^{-1}\abs{V_n}_{H^{s-1}})_n$ and $(c_n^{-1}\abs{\rho_n}_{H^{s-1}})_n$ in $l^2(\N)$. This is fortunately enough to get uniform convergence of the modal decomposition.
\end{remark}
\begin{remark}
Though mode mixing due to topography \cite{griffiths} or nonlinear terms \cite{martin} is known to occur, such a dispersive mixing does not seem to have been noticed before; the reason is that in the physical literature (e.g; \cite{Flierl,Gill}), only the hydrostatic component of the pressure is taken into account (i.e. the term $\dt w$ is neglected in the second equation of \eqref{Euler_pert_lin}).
\end{remark}
\begin{example}\label{ex2}
We represent in Table \ref{table1} the mixing coefficients $\alpha_{mn}$ for $1\leq m,n \leq 6$ for the example of stratified ocean considered in Example \ref{ex1} (the table being symmetric, we just represent the upper half). For this example, the mixing coefficients between two neighboring coefficients ($n$ and $n\pm 1$) are smaller but of same order as the corresponding diagonal terms, while the other interactions are at least one order smaller. It follows that for perturbations for which dispersion is significant (i.e. perturbations that do not have a too large wavelength), it can trigger neighboring modes. 
\begin{table}
\caption{Mixing coefficients $\alpha_{mn}$ for $1\leq m,n\leq 6$ (in ${\rm hour}^{-2}$)}
\begin{center}
\begin{tabular}{|c|c|c|c|c|c|}
0.26& -0.22& -0.02 &0.07 & -0.05 &-0.01 \\
 * & 0.57 & -0.13 & 0.01 &-0.05 &0.02  \\
 * & * & 0.45 & -0.20 & -0.03 & 0.08 \\
 * & * &* & 0.47 & 0.14 & 0.06 \\
 * & * &* &* & 0.45 & 0.17 \\
  * & * &* &* & *  & 0.38
\end{tabular}
\end{center}
\label{table1}
\end{table}%

\end{example}

\section{The shallow water regime}\label{sectSW}

We consider here some configurations of interest in oceanography, when the horizontal scale is much larger than the depth (shallow water). Under the additional {\it strong Boussinesq assumption} under which the density is assumed to be constant in Euler's equations, we are able to derive nonlinear models. We first derive in \S \ref{sectND} the dimensionless equations, that involve two parameters: $\eps$ (nonlinearity) and $\mu$ (shallowness). The local well posedness of these equations is granted by a straightforward adaptation of Theorem \ref{theorem1}, but the resulting time of existence cannot be controlled as $\mu$ gets smaller, which is by definition the case in the shallow water regime. This problem is addressed in \S \ref{sectunif}. We then propose in \S \ref{sectmodeNL} a nonlinear extension of the modal representation of the solutions used above in the linear case.
\subsection{Dimensionless equations under the strong Boussinesq assumption}\label{sectND}

We are interested here in describing the behavior of the solutions to (\ref{Euler_pert})-(\ref{BC_pert}) for configurations arising in oceanography with the propagation of internal waves. For such applications, the equilibrium density is of the form
$$
\rho_{\rm eq}(z)=\rho_0+\tilde\rho_{\rm eq}(z)
$$
with $\rho_0$ a constant and $\tilde\rho_{\rm eq}\ll \rho_0$ . Therefore, the so called {\it strong Boussinesq assumption} is generally made; it consists in neglecting the variation of the density everywhere except in the density equation and in buoyancy forces, so that the equations under considerations are
\begin{equation}
\label{Euler_pert_B}
\left\lbrace
\begin{array}{l}
\dsp (\dt U+\eps U\cdot\grad U)=-\frac{1}{\rho_{0}}\grad P-\frac{\rho}{\rho_{0}} g{\bf e_z},\\
\dsp \dt \rho+\eps U\cdot\grad \rho+w\frac{d}{dz}\rho_{\rm eq}=0,\\
\dsp \nabla_{X,z}\cdot U=0
\end{array}\right.
\end{equation}
with the boundary conditions
\begin{equation}\label{BC_pert_B}
w_{\vert_{z=-H}}=w_{\vert_{z=0}}=0.
\end{equation}

In most cases, the configurations under consideration are typically of ``shallow water'' type, meaning that horizontal scales are much larger than vertical ones. (see for instance \cite{Gill,Chelton,griffiths}). This particular setting has strong consequences on the behavior of the solutions that are more easily captured by working with a dimensionless version of the equations for which the scales are adapted to the physical phenomenon under consideration.

We therefore define the following dimensionless variables and unknown (denoted with a tilde),
$$
\begin{array}{l}
\dsp X=L\widetilde{X},\quad z=H\widetilde{z},\quad t=\frac{L}{\sqrt{gH}}\widetilde{t},\vspace{1mm}\\
\dsp V=\sqrt{gH} \widetilde{V},\quad w=\sqrt{\mu}\sqrt{gH}\widetilde{w},\quad \rho=\rho_0 \widetilde{\rho},
\quad P=\rho_0 g H \widetilde{P},
\end{array}
$$
where $L$ is a typical horizontal scale, $H$ is the total depth, $\rho_0$ the average volumic mass of sea-water, while $\sqrt{gH}$ correspond to the 
speed of the barotropic (surface) 
mode neglected here under our rigid lid assumption, 
and  $\mu>0$ is the so called ``shallowness parameter'',
$$
\mu=\frac{H^2}{L^2}.
$$
We also introduce the quantities $\rhob$ and $N^2$ defined as
$$
\rhob(z)=\frac{\rho_{\rm eq}(zH)}{\rho_0},\qquad
N^2=N^2(z)=-\frac{d}{dz}\rhob,
$$
where it is implicitly assumed that $\rho_{\rm eq}$ is a non-increasing function of $z$ (which is an obvious and classical condition for the linear stability of the equilibrium); the (depth depending) quantity $N$ is the nondimensionalized \emph{Brunt-V\"ais\"al\"a} -- or buoyancy -- frequency.

With these new variables and unknowns, the equations  (\ref{Euler_pert})-(\ref{BC_pert}) become after
dropping the tildes and separating the equations for the horizontal and vertical velocities, and in the nondimensionalized fluid domain $X\in \R^d$, $z\in (-1,0)$, 
\begin{equation}
\label{Euler_ND}
\left\lbrace
\begin{array}{l}
\dsp (\dt V+\eps U\cdot\grad V)=-\nabla P,\vspace{1mm}\\
\dsp \mu(\dt w+\eps U\cdot\grad w)=-(\dz P+\rho),\vspace{1mm}\\
\dsp \dt \rho+\eps U\cdot\grad \rho - N^2 w=0,\vspace{1mm}\\
\dsp \nabla_{X,z}\cdot U=0,
\end{array}\right.
\end{equation}
where $\nabla$ stands for the gradient operator taken with respect to the horizontal variables alone. These equations are
complemented by the boundary conditions
\begin{equation}\label{BC_ND}
w_{\vert_{z=-1}}=w_{\vert_{z=0}}=0.
\end{equation}

\subsection{Uniform well-posedness}\label{sectunif}

One can without difficulty adapt the proof of Theorem \ref{theorem1} to get an existence result for the nondimensionalized system \eqref{Euler_ND}-\eqref{BC_ND}; however, such a proof does not provide an existence time that is uniform with respect to $\mu$ in the shallow water limit, that is, when $\mu \to 0$. The following theorem provides an existence time of order $O(\eps/\sqrt{\mu})$ and which is therefore uniformly of size $O(1)$ when the perturbation are of {\it medium amplitude} ($\eps=O(\sqrt{\mu})$ and uniformly of size $O(1/\sqrt{\eps})$ in the {\it weakly nonlinear regime} $\eps=O(\mu)$.
\begin{theorem}\label{theorem3}
Assume that the Br\"unt-Vais\"al\"a frequency $N>0$ is independent of $z$ and let $k_0\in \N$, $\nu=2k_0 > d+2$. Then for all $\rho^0, U^0\in H^{\nu}(\cS)$  such that
$\grad\cdot U^0=0$ and 
$$
\begin{cases}
\forall 0\leq k\leq k_0-1,&\dz^{2k}\rho^{0}=\dz^{2k+1}V^{0}=0 \\
\forall 0\leq k\leq k_0,&\dz^{2k}w^{0}=0
\end{cases}
\quad \mbox{ at }z=-1,0,
$$
there exists $T>0$ such that for all $0<\eps,\mu \leq 1$,  there is  a unique solution $(U,\rho)\in
C([0,\frac{T}{\eps/\sqrt{\mu}}];H^{n}(\cS)^{d+2})$ to
\eqref{Euler_ND}-\eqref{BC_ND} with initial condition
$(U^0,\rho^0)$.
\end{theorem}

\begin{proof}
A key step in the proof is that $w$ and its vertical derivatives of even order vanish at the boundary. We will use the following lemma.
\begin{lemma}\label{lemmaBC}
 If $(U,\rho)$ is a smooth solution to \eqref{Euler_ND}-\eqref{BC_ND} on a time interval $[0,T]$ such that for some $k_0\in \N$ the initial data satisfy
$$
\forall 0\leq j\leq k_0,\qquad \dz^{2j}w^{0}=\dz^{2j}\rho^{0}=\dz^{2j+1}V^{0}=0 \quad \mbox{ at }z=-1,0.
$$
These cancellations are propagated by the equations on the time interval $[0,T]$. Moreover, if  $ \dz^{2k_0+2}w^{\rm in}\,_{\vert_{z=-1,0}}=0$ then $ \dz^{2k_0+2}w_{\vert_{z=-1,0}}=0$ on $[0,T]$.
\end{lemma}
\begin{proof}[Proof of the lemma]
 Let us prove first the following identities for  $0\leq j \leq k_0$, 
$$
(\dz^{2j} w)_{\vert_{z=0,-1}}=0, \qquad (\dz^{2j}\rho)_{\vert_{z=0,-1}}=0
\qquad (\dz^{2j+1}P)=0,\qquad 
\dz^{2j+1}V=0,
$$
at $z=-1,0$. Proceeding by induction, we assume that these identities hold for $0\leq j \leq k\leq k_0-1$ and want to prove that they hold for $j=k+1$.
We recall that
$$
-(\dz^2  +\mu \Delta) P=\dz \rho+ \eps \mu \grad\cdot (U\cdot \grad U)
$$
so that, after differentiating $2k+1$ times with respect to $z$,
\begin{equation}\label{int1}
- \dz^{2k+2}(\dz P +  \rho)=\eps \mu \dz^{2k+2} (U\cdot \grad w)+\eps\mu \dz^{2k+1}\nabla\cdot (U\cdot \grad V).
\end{equation}
Applying $\dz^{2k+2}$ to the equation on $w$ one also gets
$$
\mu \dt \dz^{2k+2 }w+\eps \mu \dz^{2k+2} (U\cdot \grad w)=- \dz^{2k+2}(\dz P +  \rho)
$$
which, together with \eqref{int1}, yields
\begin{align*}
\mu \dt \dz^{2k+2 }w&=\eps\mu \dz^{2k+1}\nabla\cdot (U\cdot \grad V)\\
&=\eps\mu \dz^{2k+1}\big[ V\cdot \nabla (\nabla\cdot V)+w \dz (\nabla\cdot V)+\sum_{l=1}^d \nabla V_l\cdot \partial_l V+\nabla w\cdot \dz V\big].
\end{align*}
Recalling that $\nabla\cdot V=-\dz w$, taking the trace of the above identity at $z=-1,0$ and using the induction assumption, one deduces
$$
\mu \big( \dt +\eps V\cdot \nabla +\eps (2k+1)(\dz w)) \dz^{2k+2 }w= 0 \quad\mbox{ at }\quad z=-1,0.
$$
Therefore, if $\dz^{2k+2 }w=0$ at $t=0$ on $z=-1,0$, it remains equal to $0$ for all times, which proves the first relation of the induction assumption for $j=k+1$.  We can now apply $\dz^{2k+2}$ to the equation on $\rho$ and take the trace at the boundaries to obtain
$$
\big( \dt + \eps V\cdot \nabla ) \dz^{2k+2}\rho=0 
$$
where we used the fact that $\dz^{2k+2}(N^2 w)=0$ at $z=-1,0$ for $N^2$ chosen as in the statement of the lemma and the fact that $\dz^{2j} w=0$ at $z=-1,0$ for $0\leq j\leq k+1$. This shows that the condition $ \dz^{2k+2}\rho=0$ at $z=-1,0$ propagates from the initial data and establishes the second relation of the induction assumption for $j=k+1$. Taking the trace of \eqref{int1} at $z=-1,0$ and using the results just proved yields the third relation. For the fourth one, we apply $\dz^{2k+3}$ to the equation on $V$ and take the trace at the boundaries to obtain
$$
(\dt +\eps V\cdot \nabla+\eps (2k+3)(\dz w))\dz^{2k+3}V=0 \quad \mbox{ at }z=-1,0,
$$
showing that the desired relation is propagated from the initial condition. \\
To conclude the induction, we need to show that the assumption is true for $j=0$. The condition on $w$ is obvious, and the fact that $\rho$, $\dz P$ and $\dz V$ also vanish at the boundaries can be deduced proceeding as above. The last point of the lemma is established by running the first step of the induction proof for $k=k_0+1$.
\end{proof}

Applying $\Lambda^{\nu-2k}\dz^{2k}$ ($0\leq k\leq k_0$) to the equations of
\eqref{Euler_ND}, we get
$$
\left\lbrace
\begin{array}{l}
\dsp (\dt \tilde V_k+\eps U\cdot\grad \tilde V_k)=-\nabla \tilde P_k +\eps F_{\rm h},   \vspace{1mm}\\
\dsp \mu(\dt \tilde w_k+\eps U\cdot\grad \tilde w_k)=-(\dz \tilde P_k+\tilde\rho_k)+\eps \mu F_{\rm v},\vspace{1mm}\\
\dsp \dt \tilde\rho_k+\eps U\cdot\grad \tilde\rho_k - N^2 \tilde w_k=\eps f,\vspace{1mm}\\
\dsp \nabla_{X,z}\cdot \tilde U_k=0,
\end{array}\right.
$$
where we denoted $\tilde U_k=\Lambda^{\nu-k}\dz^{2k} U$ and $\tilde
\rho_k=\Lambda^{\nu-k}\dz^{2k} \rho$, and where
\begin{align*}
F_{\rm h}=-[\Lambda^{\nu-2k}\dz^{2k},U]\cdot \grad V,\qquad F_{\rm v}=-  [\Lambda^{\nu-2k}\dz^{2k},U]\cdot \grad w,
\end{align*}
$$
f=- [\Lambda^{\nu-2k}\dz^{2k},U]\cdot \grad \rho.
$$
By the lemma, we also have the boundary conditions
$$
{{\tilde w}_k}\,_{\vert_{z=-1}}={{\tilde w}_k}\,_{\vert_{z=0}}=0.
$$
Multiplying the first equation by
$\tilde U_k$ and the second one by $\tilde \rho_k$, and integrating by
parts, we therefore get (recall that $\grad\cdot \tilde U_k=0$), 
\begin{align*}
\frac{d}{dt}\tilde E_k
&\leq \eps \Abs{F_{\rm h}}_2\Abs{\tilde V_h}_2+\eps \sqrt{\mu}\Abs{F_{\rm v}}_2\sqrt{\mu}\Abs{\tilde w_k}_2+\eps \Abs{f}_2\Abs{\tilde \rho_k}_2\\
&\leq \eps  \big( \Abs{F_{\rm h}}_2+ \sqrt{\mu}\Abs{F_{\rm v}}_2+  \Abs{f}_2\big) \tilde E_k^{1/2}
\end{align*}
where
$$
\tilde E_k=\frac{1}{2}\Abs{\tilde V_k}^2_2+\frac{1}{2}\mu \Abs{\tilde w_k}^2_2+\frac{1}{2}\Abs{\tilde\rho_k}_2^2
$$
With $\nu\geq 2t_0+2$, the commutator estimates of Appendix \ref{appcom} imply  that
$$
\Vert F_{\rm h}\Vert_2+\sqrt{\mu}\Vert F_{\rm v}\Vert_2+\Vert f\Vert_2 \lesssim
 \frac{1}{\sqrt{\mu}}  \big( \Abs{V}_{H^\nu}+ \sqrt{\mu}\Abs{w}_{H^\nu}+  \Abs{\rho}_{H^\nu}\big)^2;
$$
we note that the singular factor $\mu^{-1/2}$ comes from the components $-[\Lambda^{\nu}, w]\dz V$ of $F_{\rm h}$ and $ -[\Lambda^{\nu}, w]\dz \rho$ of $f$ in the case $k=0$.\\
Summing over $0\leq k\leq k_0$ and using the above estimates, it follows that
\begin{align*}
\frac{d}{dt}\sum_{k=0}^{k_0}\tilde E_k
&\lesssim \frac{\eps}{\sqrt{\mu}}  \big( \Abs{V}_{H^\nu}+ \sqrt{\mu}\Abs{w}_{H^\nu}+  \Abs{\rho}_{H^\nu}\big)^2 \Big(\sum_{k=0}^{k_0}\tilde E_k\Big)^{1/2}.
\end{align*}
It is therefore possible to conclude by a nonlinear Gronwall type argument if we are able to show that
$$
\big( \Abs{V}_{H^\nu}+ \sqrt{\mu}\Abs{w}_{H^\nu}+  \Abs{\rho}_{H^\nu}\big)^2 \lesssim \sum_{k=0}^{k_0}\tilde E_k,
$$
which is a direct consequence of the following lemma.
\begin{lemma}
{\bf i.} There is a constant $C$ such that for all smooth functions $V$, $w$ and $\rho$ such that 
$$
\dz^{2k+1} V_{\vert_{z=-1,0}}=\dz^{2k} w_{\vert_{z=-1,0}}=\dz^{2k} \rho_{\vert_{z=-1,0}}+0
\quad\mbox{ for }\quad k=0,\dots,k_0-1,
$$
 one has, with $u=V$, $w$ or $\rho$,
\begin{align*}
\frac{1}{C} \Vert u \Vert_{H^\nu} \leq \sum_{k=0}^{k_0} \Vert \Lambda^{\nu-2k}\dz^{2k} u \Vert_{2} \leq C \Vert u \Vert_{H^\nu}.
\end{align*}
\end{lemma}
\begin{proof}[Proof of the lemma]
 We need to show that it is possible to control $\Lambda^{\nu-2k-1}\dz^{2k+1} V$  ($1\leq k\leq k_0$) in $L^2(\cS)$ by $ \sum_{k=0}^{k_0} \Vert \Lambda^{\nu-2k}\dz^{2k} V \Vert_{2}$. This is a consequence of Lemma \ref{lemmaBC} which under the assumption made on the vanishing of the vertical derivatives of $V$ of odd order, allows one to write
\begin{align*}
\Vert \Lambda^{\nu-2k-1}\dz^{2k+1} V\Vert_2^2 &= \int_{\cS} \Lambda^{\nu-2k-1}\dz^{2k+1} V\cdot \Lambda^{\nu-2k-1}\dz^{2k+1} V,\\
&=-\int_{\cS} \Lambda^{\nu-2k}\dz^{2k} V\cdot \Lambda^{\nu-2k-2}\dz^{2k+2} V,\\
&\leq \Abs{\Lambda^{\nu-2k}\dz^{2k} V}_2 \Abs{\Lambda^{\nu-2k-2}\dz^{2k+2} V}_2.
\end{align*}
The result for $w$  and $\rho$ is proved similarly.
\end{proof}
The proof of the theorem is therefore complete.
\end{proof}
\begin{remark}
For surface waves (water waves), the existence time one obtains in the weakly nonlinear regime in shallow water --and for the asymptotic models derived under these assumptions (Boussinesq systems or the KdV and BBM equations, etc.)-- the existence time is $O(1/\eps)$, uniformly with respect to $\mu \in (0,1)$ (see for instance \cite{Lannes_NL} and references therein). This is also true for interfacial waves between two layers of fluids of different densities \cite{Lannes2} and corresponding asymptotic models (see for instance \cite{BLS,CGK} and the review \cite{Saut} and references therein). The reason why the existence time is only $O(\eps/\sqrt{\mu})$ is partly due to the fact that irrotationality cannot be used. Indeed, if the vertical vorticity were $O(\sqrt{\mu})$, i.e if we had $\dz V=\mu \nabla w +O(\sqrt{\mu})$ then the term $[\Lambda^\nu,w]\dz V$ would not be singular with respect to $\mu$ anymore. This is consistent with \cite{Castro} where it is shown that in the presence of vorticity, the water waves equations are well posed on a time scale $O(\eps/\sqrt{\mu})$ under the assumption that the vertical vorticity is $O(\sqrt{\mu})$.
\end{remark}
\subsection{Modal decomposition and nonlinear mixing}\label{sectmodeNL}

Proceeding as for the derivation of \eqref{eqbig}, one can derive from  the linear version of \eqref{Euler_ND}-\eqref{BC_ND} the following equation for $w$,
\begin{equation}\label{eqbigbis}
\dt^2\Big[\Big(\mu \Delta +\dz^2 \Big)w\Big]+N^2 \Delta w=0,
\end{equation}
where $N$  is the Br\"unt-Vais\"al\"a frequency that is now assumed to be independent of $z$. Instead of \eqref{SL}, the relevant Sturm-Liouville problem for a modal decomposition is now
\begin{equation}\label{SLNL}
\left\lbrace
\begin{array}{l}
\dsp \frac{d^2}{dz^2}{\bf
  f}_n+\frac{N^2}{c_n^2}{\bf f}_n=0,\\
\dsp  {\bf f}_n(-1)={\bf f}_n(0)=0;
\end{array}\right.
\end{equation}
for this Sturm-Liouville problem, the eigenvalues and the orthonormal (for the $L^2(N^2 {\rm d}z) $ scalar product)- basis of eigenfunctions are explicit, as well as the functions ${\bf g}_n$ defined in Lemma \ref{lemmebases} (which are now orthonormal for the $L^2$ scalar product),
$$
c_n=\frac{N}{n\pi},\qquad {\bf f}_n=\frac{\sqrt{2}}{N}\sin(n\pi z), \qquad {\bf g}_n=\sqrt{2}\cos(n\pi z).
$$
The modal decompositions \eqref{decompw} and \eqref{decomp} are consequently replaced by 
\begin{equation}\label{decompbis}
\begin{array}{lcl}
V(t,X,z)=\dsp \sum_{n=0}^\infty V_n (t,X)
{\bf g}_n(z),&
P(t,X,z)=\dsp \sum_{n=0}^\infty P_n(t,X){\bf g}_n(z),\\
w(t,X,z)= \dsp \sum_{n=0}^\infty w_n (t,X)
{\bf f}_n(z), &
\rho(t,X,z)=\dsp \sum_{n=1}^\infty \rho_n
(t,X)N^2{\bf f}_n(z).
\end{array}
\end{equation}

\begin{proposition}\label{propNL}
The smooth solutions of \eqref{Euler_ND}-\eqref{BC_ND} admit a modal decomposition of the form \eqref{decompbis} that satisfy the following coupled system of evolution equations up to terms of size $O(\eps\mu)$,
$$
\begin{cases}
\dsp \big(1-\mu \frac{1}{\pi^2 n^2}\Delta \big)\dt V_n+\frac{N}{n\pi}\nabla \rho_n=
- \eps \frac{1}{\sqrt{2}}\sum_{(p\pm q)^2=n^2} \big[ V_p\cdot \nabla V_q \pm \frac{q}{p} \nabla\cdot V_ p V_q\big],\\
\dsp \dt \rho_n+ \frac{N}{n\pi}\nabla\cdot V_n=\eps \frac{1}{N^2}\frac{1}{\sqrt{2}} \sum_{(p\pm q)^2=n^2} \pm \big[ V_p\cdot \nabla \rho_q +\frac{q}{p} \nabla\cdot V_p \rho_q\big].
\end{cases}
$$
\end{proposition}
\begin{remark}
The nonlinear terms induce a mixing of the different modes of a different nature than the dispersive mixing exhibited in the linear case when $N$ is not constant. It can be expected that after some time, the coupling becomes less and less efficient. Indeed, at first order, $V_p$ and $V_q$ ($p\neq q$) travel at a different speed, and if they are localized enough, the product $V_p\cdot \nabla V_q$ (as well as the other coupling terms) becomes very small for large times. Such an asymptotic has been used for instance in diffractive optics \cite{LannesAA}. 
\end{remark}
\begin{proof}
Decomposing $V$, $w$, $\rho$ and $P$ as in \eqref{decompbis} and taking the $L^2(-1,0)$ scalar product of the first equation of \eqref{Euler_ND} with ${\bf g}_n$ gives
$$
\dt V_n+\nabla P_n= - \eps \sum_{p=1}^\infty\sum_{q=1}^\infty \big[\beta_{pqn} V_p\cdot \nabla V_q +\gamma_{pqn} w_p \frac{V_q}{c_q}\big]
$$
with
$$
\beta_{pqn}=({\bf g}_p{\bf g}_q,{\bf g}_n), \qquad \gamma_{pqn}=-N^2 ({\bf f}_p{\bf f}_q,{\bf g}_n)
$$
Taking the $L^2(-1,0)$ scalar product of the second equation of \eqref{Euler_ND} ${\bf f}_n$ gives similarly
$$
\frac{1}{N^2}\mu \dt w_n+\rho_n- \frac{1}{c_n}P_n=O(\eps\mu)
$$
For the equation on $\rho$, we take the scalar product with ${\bf f}_n$ to obtain
$$
\dt \rho_n- w_n=\eps \frac{1}{N^2} \sum_{p=1}^\infty\sum_{q=1}^\infty \big[\gamma_{pqn} V_p\cdot \nabla \rho_q +\gamma_{pnq} w_p \frac{\rho_q}{c_q}\big]
$$
while the divergence free condition yields
$$
w_n=-c_n\nabla\cdot V_n.
$$
Up to terms of order $O(\eps\mu)$, we therefore obtain the following system on $(V_n,\rho_n)$,
$$
\begin{cases}
\dsp \big(1-\mu \frac{c_n^2}{N^2}\Delta \big)\dt V_n+c_n\nabla \rho_n=
- \eps \sum_{p=1}^\infty\sum_{q=1}^\infty \big[\beta_{pqn} V_p\cdot \nabla V_q +\gamma_{pqn}\frac{c_p}{c_q} \nabla\cdot V_ p V_q\big],\\
\dsp \dt \rho_n+ c_n\nabla\cdot V_n=\eps \frac{1}{N^2} \sum_{p=1}^\infty\sum_{q=1}^\infty \big[\gamma_{pqn} V_p\cdot \nabla \rho_q +\gamma_{pnq}\frac{c_p}{c_q} \nabla\cdot V_p \rho_q\big]
\end{cases}
$$
In order to get the result of the proposition, we just need to compute the coefficients $\beta_{pqn}$ and $\gamma_{pqn}$ using the explicit expression of the eigenfunctions derived above. One readily checks that $\sqrt{2}{\bf g}_p{\bf g}_q{\bf g}_n$ is equal to
$$
\cos\big((p+q+n)\pi z \big)+\cos\big((p+q-n)\pi z \big)+\cos\big((p-q+n)\pi z \big)+\cos\big((p-q-n)\pi z \big)
$$
Therefore one has $\beta_{pqn}=0$ except if $(p\pm q)^2=n^2$ in which case $\beta_{pqn}=\frac{1}{\sqrt{2}}$. Similarly, $\gamma_{pqn}=0$ unless $(p\pm q)^2=n^2$, in which case $\gamma_{pqn}=\pm \frac{1}{\sqrt{2}}$. The proposition follows.

\end{proof}

\section{Perspectives}\label{sectperspectives}

Let us briefly mention here some interesting perspective for further works. An obvious one would be to take into account Coriolis effects as they become relevant at large oceanic scales. In particular, the dispersive effects due to the Earth rotation are known to be important (see for instance \cite{Chelton}) and it would be interesting to see how they compare to the dispersive effects investigated in this paper. In the same vein, taking into account thermal effect, salinity, etc. are important for many physical applications. For atmospheric studies, it may also be relevant to relax the incompressibility assumption; the analysis of the Sturm-Liouville decomposition is then complicated by the interaction with acoustic modes \cite{BK}.

\medbreak

Another interesting generalization is to consider a non zero background current, i.e. to take a non zero $U_{\rm eq}$ in \eqref{equil}. In this context, and in the shallow water limit, Maslowe and Redekopp \cite{MR} showed that it is possible to construct approximate solutions concentrated on the first eigenmode and that the corresponding coefficient satisfies a nonlinear KdV equation. This is in sharp contrast with the nonlinear Boussinesq type system derived in Proposition \ref{propNL}. Indeed, in this latter, a mode $n$ can only interact nonlinearly with modes $p$ and $q$ such that $(p\pm q)^2=n^2$. In particular, self quadratic interaction of the coefficient of the $n$-th mode cannot occur. This does not contradict \cite{MR} since one can check in that reference that the coefficient in front of the nonlinear term in the KdV equation vanishes when the background current is taken equal to $0$. This suggests however that including background currents leads to additional interesting mathematical and physical phenomenons. 

\medbreak

Let us mention a last important perspective. As said in the introduction, when the stratification is continuous but varies very rapidly in a thin layer, two layers models are used to describe internal waves: instead of a single continuously stratified fluid, one considers two fluids of different densities separated by an interface. The propagation of internal waves therefore reduces to the study of the evolution of this interface. The convergence of internal waves with a sharp continuous background stratification towards the solution of the corresponding two-fluids models is therefore a natural question. It was answered by the affirmative in \cite{James} in the particular case of solitary waves. The general case of non steady waves is much more complex because of the presence of Kelvin-Helmholtz instabilities created by the discontinuity of the tangential velocity at the interface. In particular, two-fluids Euler equations are ill posed \cite{Ebin,IguchiTani,KamotskiLebeau}; including surface tension effects, well-posedness is restored and a generalized Rayleigh-Taylor criterion governing well-posedness can be derived \cite{Lannes2}. The interest of this last result is that it shows that a very small enough of surface tension is enough to stabilize interfacial waves, but its drawback is that, in the context of internal waves, there is no natural definition of surface tension. It is natural to conjecture that an alternative mechanism for the control of Kelvin-Helmholtz instabilities could come from the sharp but continuous variation of the density at the "interface" (indeed, as shown here, continuously stratified models are locally well posed). In order to understand this mechanism, a first step is to study the behavior of the Sturm-Liouville modal decomposition as the background stratification $\rho_{\rm eq}$ converges to a discontinuous stratification.\\
To be more precise, let $\alpha \in C^\infty(\R)$ be a positive function, compactly supported in $(0,1)$ and such that $\int_\R \alpha =1$, and define the smoothed jump function $\chi$ as $\chi(x)=\int_{-\infty}^x \alpha$. We consider here a strip with height $H=1$ and a family of  continuous stratifications $(\rho_{{\rm eq},\delta})_{0<\delta<1}$ that converges as $\delta \to 0$ to a discontinuous stratification with jump located at $z=z_0\in (-1,0)$ with density $\rho_+$ in the lower layer and $\rho_-$ in the upper one,
\begin{equation}\label{densitysharp}
\rho_{{\rm eq},\delta}(z)=\rho_{+}\exp(-\delta z)-(\rho_+-\rho_-)\chi(\frac{z-z_0}{\delta})\qquad (\rho_+>\rho_-),
\end{equation}
for all $z\in [-1,0]$, and we consider the associated Sturm-Liouville problem
\begin{equation}\label{SLdelta}
\frac{1}{\rho_{{\rm eq},\delta}}\frac{d}{dz}\big( \rho_{{\rm eq},\delta}\frac{d}{dz} f_{n,\delta}\big)+\frac{N^2}{c_{n,\delta}^2} f_{n,\delta}=0, \qquad (-1< z<0),
\end{equation}
with the usual boundary conditions $f_{n,\delta}(-1)=f_{n,\delta}(0)=0$.
The following proposition provides the asymptotic behavior of the first eigenvalue and of the associated eigenmode as $\delta \to 0$.
\begin{proposition}
Let $\underline{c}$ and $\underline{f}$ be defined as
$$
\underline{c}^2=(\rho_+-\rho_-) g \Big( \frac{\rho_+}{(z_0+1)}+\frac{\rho_-}{(-z_0)} \Big)^{-1}
$$
and
$$
\underline{f}(z)=\begin{cases}
a(z+1)& \mbox{ if }-1\leq z \leq z_0,\\
a \frac{1+z_0}{z_0}z & \mbox{ if }  z_0 \leq z\leq 0,
\end{cases}
\quad\mbox{ with }\quad a=\frac{1}{\sqrt{(\rho_+-\rho_-) g }(z_0+1)}.
$$
For all $\delta\in (0,1)$, let also $c_{1,\delta}^{-2}$ be the smallest eigenvalue of the Sturm-Liouville problem \eqref{SLdelta} and denote by $f_{1,\delta}$ the associated unit eigenfunction such that $f'_{1,\delta}(-1)>0$. Then, as $\delta \to 0$, the following approximations hold
$$
c_{1,\delta}^2=\underline{c}^2+O(\delta)
\quad\mbox{ and }\quad
\abs{f_{1,\delta}-\underline{f}}_{L^\infty(-1,0)}=O(\delta).
$$
\end{proposition}
\begin{proof}
{\bf Step 1}. Let us first prove that $\frac{1}{c_{1,\delta}^2}\leq \frac{1}{\underline{c}^2}+O(\delta)$. We recall that $c_{1,\delta}^2$ is given by the Rayleigh quotient
$$
\frac{1}{c_{1,\delta}^2}=\inf_{f\in C_{\rm pw}^1([-1,0]),f(-1)=f(0)=0} \frac{\int_{-1}^0 \rho_{{\rm eq},\delta} (f')^2 }{\int_{-1}^0 \rho_{{\rm eq},\delta} N^2_\delta f^2 },
$$
where $C^1_{\rm pw}([-1,0])$ denotes the space of continuous and piecewise $C^1$ functions on $[-1,0]$ and  $N_\delta^2=-g \rho'_{{\rm eq},\delta}/\rho_{{\rm eq},\delta}$. An upper bound for $c_{1,\delta}^2$ is therefore obtained by evaluating the quotient $R_\delta(f)$ in the right-hand-side with $f=\underline{f}$ given as in the statement of the proposition
(the amplitude coefficient $a$ is chosen so that $\int_{-1}^0 \rho_{{\rm eq},\delta} N^2_\delta \underline{f}^2 \to 1$ as $\delta \to 0$). Using the fact that $\frac{1}{\eps}\chi'(\frac{z-z_0}{\eps})$ is an approximation of unity converging to the Dirac distribution centered at $z=z_0$, one readily checks that $R(\underline{f})=\underline{c}^{-2}+O(\delta)$, which proves the result.\\
{\bf Step 2}. Let $f_{1,\delta}$ be a unit eigenfunction associated to the eigenvalue $1/c_{1,\delta}^2$ of the Sturm-Liouville problem. From the definition of $\rho_{{\rm eq},\delta}$,  $f_{1,\delta}$ solves the following ODE on $(-1,z_0-\delta)$ and $(z_0+\delta,1)$,
$$
f_{1,\delta}''-\delta f_{1,\delta}'+g \delta  f_{1,\delta}=0,
$$
from which we can deduce that
$$
f_{1,\delta}(z)=a_\delta (z+1)+O(\delta) \mbox{ on } [-1,z_0-\delta]
\quad\mbox{ and }\quad
f_{1,\delta}(z)=b_\delta z +O(\delta ) \mbox{ on } [z_0+\delta,1]
$$
with $a_\delta=O(1)$ and $b_\delta=O(1)$ as $\delta \to 0$. On the segment $[z_0-\delta,z_0+\delta]$ we can write
$$
f_{1,\delta}(z)=f_{1,\delta}(z_\pm)+(z-z_\pm)f_{1,\delta}'(z_\pm)+\int_{z_\pm}^z \frac{1}{\rho_{{\rm eq},\delta}}\int_{z_\pm}^{z'} \big( \rho_{{\rm eq},\delta} f_{1,\delta}' (z'')\big)' {\rm d}z''{\rm d}z',
$$
with $z_\pm =z_0\pm \delta$. Using the Sturm-Liouville equation to simplify the last integral, one gets
$$
f_{1,\delta}(z)=f_{1,\delta}(z_\pm)+(z-z_\pm)f_{1,\delta}'(z_\pm)+\int_{z_\pm}^z \frac{1}{\rho_{{\rm eq},\delta}}\int_{z_\pm}^{z'}  \frac{g}{c_{1,\delta}^2} \rho_{{\rm eq},\delta}'(z''){\rm d}z''{\rm d}z'.
$$
Using the definition of $\rho_{{\rm eq},\delta}$ and the upper bound in $1/{c_{1,\delta}}^2$ derived in Step 1, we obtain that on  $[z_0-\delta,z_0+\delta]$ one has
$$
f_{1,\delta}(z)=a_\delta(1+z_0)+O(\delta)=b_\delta z_0+O(\delta),
$$
so that $a_\delta=a+O(\delta)$, $b_\delta=a\frac{1+z_0}{z_0}+O(\delta)$ and the result follows.
\end{proof}

The behavior predicted by the proposition can be checked on Figure \ref{subfig2b} that represents the speeds associated to the first modes of the modal decomposition associated to \eqref{densitysharp} (with $\delta=0.0005$ and $z_0=-1/3$, in which case the formula for the asymptotic speed gives $\underline{c}\approx 0.209$) and on Figure \ref{fig5} where the first 6 modes are represented. A striking fact is that the formula for the asymptotic speed $\underline{c}$, which is valid even in non shallow water configurations, coincides with the formula for the propagation of linear waves in two-fluid models in shallow water \cite{BLS}. Moreover, for such two layers models, the vertical velocity has a linear dependence in $z$ and is continuous at the interface: it is therefore a multiple of the asymptotic eigenmode $\underline{f}$. It is therefore natural to conjecture that one could find the two-layers shallow water model as a double shallow water/sharp stratification limit by focusing on the first mode of the Sturm-Liouville decomposition and showing that the contribution of the other modes is negligible. This is consistent with the fact that, for surface water waves, the shallow water limit is somehow a "low frequency" limit, and this would also allow one to bypass the two-fluids Euler equations and the corresponding difficulties raised by Kelvin-Helmholtz instabilities. Indeed, such instabilities do not appear for the two-fluid shallow water equations if the discontinuity of the tangential velocity at the interface is small enough \cite{GLS,BR}.
\begin{figure}
\includegraphics[width=0.7\textwidth]{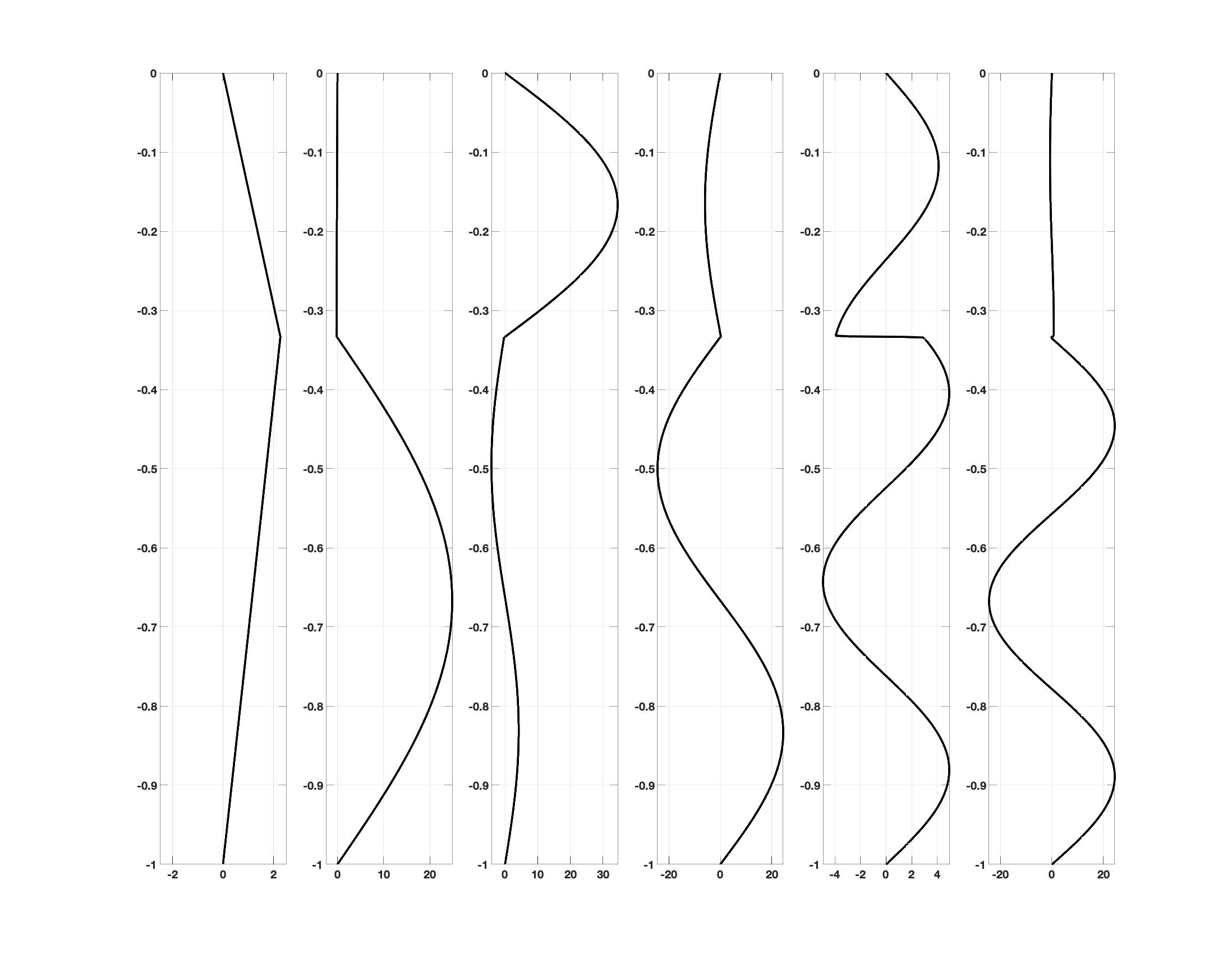}
\caption{The first modes associated to the stratification \eqref{densitysharp} with $\delta=0.0005$ and $z_0=-1/3$}
\label{fig5}
\end{figure}

\appendix

%\section{Proof of Lemma \ref{lemmasecular}}
%
%\subsection{The one-dimensional case}
%
%When $d=1$, (\ref{systsecular}) is equivalent to
%$$
%\left\lbrace
%\begin{array}{l}
%\dsp (\dt + c\dx)(\rho+cV)=f+cG,\\
%\dsp (\dt-c\dx)(\rho-cV)=f-cG.
%\end{array}\right.
%$$
%\subsection{The two-dimensional case}
%
%When $d=2$, we can decompose the wave system under the form
%$$
%\sum_{j=\pm,0}\big(\dt+i\lambda_j(D)\big)\Pi_j(D)\left(\begin{array}{c} \rho \\ c_n V\end{array}\right)=\left(\begin{array}{c} QQQ\\QQQ\end{array}\right)
%$$
%\appendix

\section{Product and commutator estimates}\label{appcom}

We use here the notations
$$
A=B+\langle C\rangle_{s>s_0}
$$
to mean that $A=B$ if $s\leq s_0$, and $A=B+C$ otherwise.

\medbreak

From the following classical product estimate for functions on $\R^d$,
$$
\forall t_0>d/2,\quad \forall s\geq 0,\qquad \abs{fg}_{H^s}\lesssim
\abs{f}_{H^s}\abs{g}_{H^{t_0}}+\Big\langle \abs{f}_{H^{t_0}}\abs{g}_{{H^s}} \Big\rangle_{s>t_0},
$$
we deduce from $L^\infty\times L^2$ product estimates on the strip
$\cS$ and the continuous embedding $H^{s+1/2,1}\subset L^\infty
H^{s}$ ($s\in \R$) that functions on the strip $\cS$ satisfy the
following product estimate, for all $t_0>d/2$, and $s\geq 0$,
\begin{equation}\label{estprod}
\Abs{FG}_{H^{s,0}}\lesssim
\Abs{F}_{H^{s,0}}\Abs{G}_{H^{t_0+1/2,1}}+\Big\langle \Abs{F}_{H^{t_0+1/2,1}}\Abs{G}_{{H^{s,0}}} \Big\rangle_{s>t_0},
\end{equation}
It is possible to deduce another product estimate from the product
estimate on $\R^d$; for instance
\begin{equation}\label{estprodb}
\Abs{FG}_{H^{s,0}}\lesssim
\Abs{F}_{H^{s,0}}\Abs{G}_{H^{t_0+1/2,1}}+\Big\langle \Abs{F}_{H^{t_0,0}}\Abs{G}_{{H^{s+1/2,1}}} \Big\rangle_{s>t_0};
\end{equation}
in particular, if $r_0=t_0$ if $s\leq t_0$ and $r_0=s$ otherwise, then
\begin{equation}\label{estprodbis}
\Abs{FG}_{H^{s,0}}\lesssim
\Abs{F}_{H^{r_0,0}}\Abs{G}_{H^{r_0+1/2,1}}.
\end{equation}
We also need in this paper product estimates in $H^{s,k}$.
Remarking
that
\begin{align*}
\Abs{FG}_{H^{s,k}}\lesssim \sum_{k'+k''=k} \Abs{(\dz^{k'}
  F)(\dz^{k''}G)}_{H^{s-k,0}},
\end{align*}
we easily deduce from \eqref{estprodb} that
\begin{equation}\label{prodalg}
\forall t_0>d/2,\quad \forall s\geq 2t_0+1,\quad\forall 0\leq k\leq
s,\qquad
\Abs{FG}_{H^{s,k}}\lesssim \Abs{F}_{H^{s,k}}\Abs{G}_{H^{s,k}}.
\end{equation}

\medbreak

For commutators with horizontal derivatives, we first recall the following estimate that combines the
Kato-Ponce estimates (for large $s$) and Coifmann-Meyer estimates (for
small $s$), for functions defined over $\R^d$,
$$
\forall t_0>d/2,\quad \forall s\geq 0,\qquad \babs{[\Lambda^s,f]g}_{2}\lesssim
\abs{f}_{H^{t_0+1}}\abs{g}_{H^{s-1}}+\Big\langle \abs{f}_{H^s}\abs{g}_{H^{t_0}}\Big\rangle_{s>t_0+1}
$$
(see for instance Theorems 3 and 6 in \cite{Lannes_JFA}). 
Using the continuous embedding $H^{s+1/2,1}\subset L^\infty
H^{s}$ ($s\in \R$), one readily deduces the following estimate for
functions defined in the strip $\cS$, and for all $t_0>d/2$, and $s\geq 0$
\begin{equation}\label{estcom}
\bAbs{[\Lambda^s,F]G}_{2}\lesssim
\Abs{F}_{H^{t_0+3/2,1}}\Abs{G}_{H^{s-1,0}}+\Big\langle \Abs{F}_{H^{s,0}}\Abs{G}_{H^{t_0+1/2,1}}\Big\rangle_{s>t_0+1}
\end{equation}
or also
\begin{equation}\label{estcombis}
\bAbs{[\Lambda^s,F]G}_{2}\lesssim
\Abs{F}_{H^{t_0+3/2,1}}\Abs{G}_{H^{s-1,0}}+\Big\langle \Abs{F}_{H^{s+1/2,1}}\Abs{G}_{H^{t_0,0}}\Big\rangle_{s>t_0+1}
\end{equation}
(see for instance \S B.2.2. in \cite{Lannes_book}).

Let us finally consider commutators with horizontal {\it and} vertical
derivatives, namely, with $\Lambda^{s-k}\dz^k$ ($k=1,2$), we first
remark that
$$
[\Lambda^{s-k}\dz^k,F]G=\Lambda^{s-k}
[\dz^k,F]G+[\Lambda^{s-k},F]\dz^k G.
$$
Using the product estimates \eqref{estprod}-\eqref{estprodb} for the
first term of the right-hand-side, and the commutator estimates
\eqref{estcom}-\eqref{estcombis} for the second one, we get
\begin{equation}\label{estcomf}
\forall n\geq 2t_0+2,\quad\forall 0\leq k\leq
n,\qquad  \bAbs{[\Lambda^{n-k}\dz^k,F]G}_2\lesssim \Abs{F}_{H^{n-1}}\Abs{G}_{H^{n-1}}.
\end{equation}

\begin{acknowledgements}
D. L. wants to thank J.-F. Bony, N. Popoff and B. Young for fruitful discussions about this work.
\end{acknowledgements}

\section*{Conflict of interest}
 The authors declare that they have no conflict of interest.


\begin{thebibliography}{}
 \bibitem{BarrosChoi} \textsc{R. Barros and W. Choi}, {\it Inhibiting shear instability induced by large amplitude internal solitary waves in two-layer flows with a free surface}, Stud. Appl. Math. {\bf 122} (2009), 325-346.
 
 \bibitem{BarrosChoi2} \textsc{R. Barros and W. Choi}, {\it On regularizing the strongly nonlinear model for two-dimensional internal waves}, Physica D {\bf 264} (2013), 27-34.
 
 \bibitem{Ben2} \textsc{T. B. Benjamin}, {\it Internal waves of finite amplitude and permanent form},  J. Fluid Mech., {\bf 25} (1966), 241-270.
 \bibitem{Ben} \textsc{T. B. Benjamin}, {\it Internal waves of permanent form in fluids of great depth},   J. Fluid Mech., {\bf 29} (1967), 559--592.
 \bibitem{BLS} \textsc{J. L. Bona, D. Lannes, J.-C. Saut}, {\it
     Asymptotic models for internal waves}, J. Math. Pures Appl., {\bf
     89} (2008), 538-566.
     
 \bibitem{BK} \textsc{D. Bresch, R. Klein},   {\it Spectral analysis of the compressible Euler equations for atmospheric mescales}, in preparation.

\bibitem{BR} \textsc{D. Bresch and M. Renardy}, {\it Well-posedness of two-layer shallow-water flow between two horizontal rigid plates}, Nonlinearity {\bf 24} (2011), 1081-1088.
     \bibitem{BC}\textsc{D.J. Brown and D.R. Christie}, {\it Fully nonlinear internal waves in continuously stratified Boussineq fluids}, Phys. Fluids {\bf 10} (1998), 2569-2586.
     \bibitem{Castro} \textsc{A. Castro and D. Lannes}, {\it Well-posedness and shallow-water stability for a new Hamiltonian formulation of the water waves equations with vorticity}, Indiana Univ. Math. J., {\bf 64} (2015), 1169-1270.
    \bibitem{Chelton} \textsc{D. B. Chelton, R. A. DeSzoeke, M. G. Schlax, K. El Naggar, and N. Siwertz}, {\it Geographical variability of the first baroclinic Rossby radius of deformation}, Journal of Physical Oceanography {\bf 28} (1998), 433-460.
     \bibitem{CGK}\textsc{W. Craig, P. Guyenne and H. Kalisch}, {\it Hamiltonian long-wave expansions for free surfaces and
interfaces}, Comm. Pure. Appl. Math. {\bf 58} (2005)1587-1641.


\bibitem{Danchin} \textsc{R. Danchin}, {\it On the well-posedness of
    the incompressible density-dependent Euler equations in the $L^p$
    framework}, J. Differential Equations {\bf 248} (2010),
  2130--2170.
  \bibitem{DA}\textsc{R.E. Davis and A. Acrivos}, {\it Solitary waves in deep water},  J. Fluid Mech., {\bf 29} (1967), 593-601.
  \bibitem{DJ}\textsc{M.L. Dubreil-Jacotin}, {\it Sur les th\' eor\`emes d'existence relatifs aux ondes p\' eriodiques dans les liquides h\' et\' erog\`enes}, J. Math. Pures Appl. {\bf 16} (1937), 43-
  
  \bibitem{Duchene} \textsc{V. Duch\^ene, S. Israwi and R. Talhouk}, {\it A new class of two-layer Green-Naghdi systems with improved frequency dispersion}, Stud. Appl. Math. {\bf 137} (2016).
  
  
 \bibitem{Ebin} \textsc{G. Ebin}, {\it Ill-posedness of the Rayleigh-Taylor and Kelvin-Helmotz problems for incompressible fluids}, Commun. Partial Differ. Equ. {\bf 13} (1988), 1265-1295.
  
\bibitem{Feliks} \textsc{Y. Feliks, M. Ghil, and E. Simonnet}, {\it Low-frequency variability in the midlatitude atmosphere induced by an oceanic thermal front}, Journal of the atmospheric sciences {\bf 61} (2004), 961-981.
  
  
 \bibitem{Flierl} \textsc{ G. R. Flierl},  {\it Models of vertical structure and the calibration of two-layer models}, Dynamics of Atmospheres and Oceans {\bf 2} (1978), 341-381.
  
  
\bibitem{Fulton}  \textsc{C. T. Fulton and S. A. Pruess}, {\it Eigenvalue and eigenfunction asymptotics for regular Sturm-Liouville problems}, Journal of Mathematical Analysis and Applications {\bf 188} (1994), 297-340.



\bibitem{Gill}  \textsc{A. E. Gill}, {\it Atmosphere-Ocean Dynamics}, volume 30 of International Geophysics Series. Academic Press, 1982.

\bibitem{griffiths} \textsc{S. D. Griffiths, R. H. Grimshaw},  {\it  Internal tide generation at the continental shelf modeled using a modal decomposition: Two-dimensional results}, Journal of Physical Oceanography {\bf 37} (2007), 428-451.

\bibitem{GLS} \textsc{P. Guyenne, D. Lannes, and J.-C. Saut}, {\it Well-posedness of the Cauchy problem for models of large amplitude internal waves}, Nonlinearity {\bf 23} (2010), 237.

  \bibitem{HM} \textsc{K.R. Helfrich and W.K. Melville},
{\it Long nonlinear internal waves}, Annual Review of Fluid
Mechanics {\bf 38} (2006) 395-425.


\bibitem{IguchiTani} \textsc{T. Iguchi, N. Tanaka, A. Tani}, {\it On the two-phase free boundary problem for two- dimensional water waves}, Math. Ann. {\bf 309} (1997), 199-223.




\bibitem{Itoh} \textsc{S. Itoh}, {\it Cauchy problem for the Euler
    equations of a nonhomogeneous ideal incompressible fluid. II},
  J. Korean Math. Soc. {\bf 32} (1995), 41--50.

\bibitem{ItohTani} \textsc{S. Itoh, A. Tani}, {\it Solvability of
    nonstationary problems for nonhomogeneous incompressible fluids
    and the convergence with vanishing viscosity}, Tokyo J. Math. {\bf
    22} (1999), 17--42.
    
 \bibitem{James}\textsc{   G. James}, {\it Internal travelling waves in the limit of a discontinuously stratified fluid}, Archive for rational mechanics and analysis {\bf 160} (2001), 41-90.
    
 \bibitem{KamotskiLebeau}   \textsc{V. Kamotski, G. Lebeau}, {\it On 2D Rayleigh-Taylor instabilities}, Asymptot. Anal. {\bf 42} (2005), 1-27.
    
    
\bibitem{Klein} \textsc{R. Klein, U. Achatz, D. Bresch et al.}, {\it Regime of validity of soundproof atmospheric flow models}, Journal of the Atmospheric Sciences {\bf 67} (2010), 3226-3237.
    
    
    \bibitem{KKD}\textsc{T. Kubota, D.R.S Ko and L.D. Dobbs}, {\it Weakly nonlinear, long internal gravity waves in stratified fluids of finite depth}, J. Hydronautics {\bf 12} (1978), 157-165.
    
    \bibitem{LannesAA} \textsc{D. Lannes}, {\it Dispersive effects for nonlinear geometrical optics with rectification}, Asymptotic Analysis {\bf 18} (1998), 111-146.
\bibitem{Lannes_JFA} \textsc{D. Lannes}, {\it Sharp estimates for
    pseudo-differential operators with symbols of limited smoothness
    and commutators}, J. Funct. Anal. {\bf 232} (2006), 495--539.
\bibitem{Lannes_book} \textsc{D. Lannes}, {\it The Water Waves
    Problem: Mathematical Analysis and Asymptotics}, volume 188 of
  Mathematical Surveys and Monographs. AMS, 2013.
  \bibitem{Lannes2} \textsc{D. Lannes}, {\it A stability criterion for two-fluid interfaces and applications},
Arch. Rational Mech. Anal. {\bf 208} (2013) 481-567.
\bibitem{Lannes_NL} \textsc{D. Lannes}, Modeling shallow water waves, submitted.
\bibitem{LannesMing} \textsc{D. Lannes, M. Ming}, {\it The Kelvin-Helmholtz instabilities in two-fluids shallow water models}, In Hamiltonian Partial Differential Equations and Applications, volume 75 of Fields Institute Communications. Springer-Verlag New York, 2015.
\bibitem{Long}\textsc{R.R. Long}, {\it Some aspects of the flow of stratified fluids.I. A theoretical investigation}, Tellus {5} (1953), 42-
\bibitem{Long2}\textsc{R.R. Long}, {\it On the Boussinesq approximation and its role in the theory of internal waves}, Tellus {\bf 17} (1965), 46.

\bibitem{martin} \textsc{S. Martin, W. Simmons, C. Wunsch}, {\it The excitation of resonant triads by single internal waves}, Journal of Fluid Mechanics {\bf 53} (1972), 17-44.

\bibitem{MR} \textsc{S. Maslowe and L. Redekopp}, {\it Long nonlinear waves in stratified shear flows}, Journal of Fluid Mechanics {\bf 101} (1980), 321-348.

\bibitem{Ono}\textsc{H. Ono}, {\it Algebraic solitary waves in stratified fluids}, J. Phys. Soc. Jpn. {\bf 39} (1975), 1082-1091.
\bibitem{Saut} \textsc{J. -C. Saut}, {\it Asymptotic models for surface and internal waves}, IMPA, 2013.
\bibitem{Yih}\textsc{C.S. Yih}, {\it Gravity waves in a stratified fluid}, J. Fluid Mech. {\bf 8} (1960), 48-
\bibitem{Yih2}\textsc{C.S. Yih}, {\it Exact solutions for steady two-dimensional flow of a stratified fluid}, J. Fluid Mech. {\bf 9} (1960), 161-




 \end{thebibliography}


\end{document}